\documentclass[10pt,fleqn]{article}
\usepackage{amsmath,amsthm,amssymb,amscd}
\usepackage{graphicx}
\usepackage{graphics}
\usepackage{verbatim}
\usepackage{enumerate}
\usepackage{mathrsfs}
\usepackage{color}
\usepackage[english]{babel}
\makeatletter

\newcommand{\Rmnum}[1]{\expandafter\@slowromancap\romannumeral #1@}

\newtheorem{theorem}{Theorem}[section]

\newtheorem{problem}[theorem]{Problem}
\newtheorem{lemma}[theorem]{Lemma}

\newtheorem{claim}{Claim}[theorem]

\newtheorem{remark}[theorem]{Remark}
\makeatother

\begin{document}
	
	\title{(1,0,0)-colorability of planar graphs without cycles of length 4 or 6}
	
	\vspace{3cm}
	\author{Ligang Jin\footnotemark[1], Yingli Kang\footnotemark[2], Peipei Liu\footnotemark[3], Yingqian Wang\footnotemark[1]}
	\footnotetext[1]{Department of Mathematics,
		Zhejiang Normal University, Yingbin Road 688,
		321004 Jinhua,
		China; 
		ligang.jin@zjnu.cn, yqwang@zjnu.cn}
	\footnotetext[2]{Department of Mathematics, Jinhua Polytechnic, Western Haitang Road 888, 321017 Jinhua, China; ylk8mandy@126.com}
	\footnotetext[3]{Hangzhou Weike software engineering Co. Ltd., Western Wenyi Road 998, 310012 Hangzhou, China; 784873860@qq.com}
		
	\date{}
	
	\maketitle

\begin{abstract}
A graph $G$ is $(d_1,d_2,d_3)$-colorable if the vertex set $V(G)$ can be partitioned into three subsets $V_1,V_2$ and $V_3$ such that for $i\in\{1,2,3\}$, the induced graph $G[V_i]$ has maximum vertex-degree at most $d_i$. So, $(0,0,0)$-colorability is exactly 3-colorability.
	
The well-known Steinberg's conjecture states that every planar graph without cycles of length 4 or 5 is 3-colorable. 
As this conjecture being disproved by Cohen-Addad etc. in 2017, 
a similar question, whether every planar graph without cycles of length 4 or $i$ is 3-colorable for a given $i\in \{6,\ldots,9\}$, is gaining more and more interest.
In this paper, we consider this question for the case $i=6$ from the viewpoint of improper colorings. More precisely, we prove that every planar graph without cycles of length 4 or 6 is (1,0,0)-colorable,
which improves on earlier results that they are (2,0,0)-colorable and also (1,1,0)-colorable, and on the result that planar graphs without cycles of length from 4 to 6 are (1,0,0)-colorable.
\end{abstract}

\textbf{Keywords:}
planar graphs, (1,0,0)-colorings, cycles, discharging, super-extension

\section{Introduction}
The graphs considered in this paper are finite and simple. 
A graph is planar if it is embeddable into the Euclidean plane. A plane graph $(G,\Sigma)$ is a planar graph $G$ together with an embedding $\Sigma$ of $G$ into the Euclidean plane, that is, $(G,\Sigma)$ is a particular drawing of $G$ in the Euclidean plane.
In what follows, we will always say a plane graph $G$ instead of $(G,\Sigma)$, which causes no confusion since in this paper no two embeddings of the same graph $G$ will be involved in.  

In the field of 3-colorings of planar graphs, one of the most active topics is about a conjecture proposed by Steinberg in 1976: 
every planar graph without cycles of length 4 or 5 is 3-colorable.
There had been no progress on this conjecture for a long time, until Erd\"{o}s \cite{Steinberg1993211} suggested a relaxation of it:
does there exist a constant $k$ such that every planar graph without cycles of length from 4 to $k$ is 3-colorable?
Abbott and Zhou \cite{AbbottZhou1991203} confirmed that such $k$ exists and $k\leq 11$.
This result was later on improved to $k\leq 9$ by
Borodin \cite{Borodin1996183} and, independently, by Sanders and Zhao \cite{SandersZhao199591}, and to
$k\leq 7$ by Borodin etc. \cite{4567}.
Steinberg's conjecture was recently disproved by Cohen-Addad etc. \cite{CA-disprove}. 
Hence, associated to Erd\"{o}s' relaxation, only one question remains unsettled.
\begin{problem}\label{pro_4to6}
	Is it true that planar graphs without cycles of length from 4 to 6 are 3-colorable?
\end{problem}

A more general problem than Steinberg's Conjecture was formulated in \cite{479}:
\begin{problem} \label{pro_4-and-i}
What is the maximal subset $\cal{A}$ of $\{5,6,\cdots,9\}$ such that for $i\in \cal{A}$, every planar graph with cycles
of length neither 4 nor $i$ is 3-colorable?
\end{problem}
The refutal of Steinberg's Conjecture shows that $5\notin \cal{A}$.  
For any other $i$, the question whether $i\in \cal{A}$ is still unsettled.
In this paper, we consider such question for the case $i=6$, i.e., the question whether every planar graph without cycles of length 4 or 6 is 3-colorable.

Let $d_1,d_2$ and $d_3$ be non-negative integers. A graph $G$ is $(d_1,d_2,d_3)$-colorable if the vertex set $V(G)$ can be partitioned into three subsets $V_1,V_2$ and $V_3$ such that for $i\in\{1,2,3\}$, the induced graph $G[V_i]$ has maximum vertex-degree at most $d_i$. 
The associated coloring, assigning the vertices of $V_i$ with the color $i$ for $i\in\{1,2,3\}$, is an improper coloring, a concept which allows adjacent vertices to receive the same color. Clearly, $(0,0,0)$-colorability is exactly 3-colorability. Improper coloring is a relaxation of proper coloring, providing us a way to approach the solution to some hard conjectures.
It has been combined with many different kinds of colorings of graphs, such as improper $k$-colorings, improper list colorings, improper acyclic colorings and so on.

The coloring of planar graphs gain particular attention. 
There are a serial of known results on the $(d_1,d_2,d_3)$-colorability of planar graphs, motivated by Steinberg's conjecture. For example,
Cowen etc. \cite{Cowen} proved that planar graphs are
$(2,2,2)$-colorable. 
Xu \cite{A35-111} showed that planar graphs with
neither adjacent triangles nor cycles of length 5 are
$(1,1,1)$-colorable.
So far, the best known results for planar graphs having no cycles of length 4 or 5 are that, they are (1,1,0)-colorable \cite{45-110-Hill,45-110-Xu} and also (2,0,0)-colorable \cite{45-200}, improving on some results in \cite{45-300,A35-111}. Because of the refutal of Steinberg's conjecture, the following question is the only one in this direction that remains open.
\begin{problem}
	Is it true that planar graphs having no cycles of length 4 or 5 are $(1,0,0)$-colorable?
\end{problem}

Analogously, for planar graphs having no cycles of length 4 or 6, it is known that they are (1,1,0)-colorable \cite{456-100,46-300and110} and also (2,0,0)-colorable \cite{46-200}. 
In this paper, we prove that they are further (1,0,0)-colorable, which improves on these two results.
\begin{theorem} \label{thm45tri7}
	Planar graphs with neither 4-cycles nor 6-cycles are (1,0,0)-colorable.
\end{theorem}
 Towards Problem \ref{pro_4to6}, Wang etc. \cite{456-100} shown that planar graphs having no cycles of length from 4 to 6 are $(1,0,0)$-colorable. Theorem \ref{thm45tri7} improves on this result as well.
To our best knowledge, Theorem \ref{thm45tri7} is the first result on $(1,0,0)$-colorability of planar graphs with neither 4-cycles nor $i$-cycles for $i\in\{5,6,7,8,9\}$, motivated by Problem \ref{pro_4-and-i}.

The proof of this main result uses discharging method for improper colorings. 
In Section \ref{sec_notations}, we formulate a proposition that is stronger than Theorem \ref{thm45tri7}, namely super-extended theorem. Section \ref{sec_proof} addresses the proof of the super-extended theorem, which consists of two parts: reducible configurations and discharging procedure.
For more information on discharging method, we refer to \cite{West2013,458,469}.

\section{Super-extended theorem} \label{sec_notations}
Let $G$ be a plane graph. For a set $S$ such that $S\subseteq V(G)$ or $S\subseteq E(G)$, let $G[S]$ denote the subgraph of $G$ induced by $S$.
Let $C$ be a cycle of $G$. Denote by $int(C)$ (resp. $ext(C)$) the set of vertices lying inside (resp. outside) $C$. 
Let $H$ be a subgraph of $G$ whose edges lie inside $C$ (ends on $C$ allowed) and let $H_0=H-V(C)$, such that $d_H(v)=3$ for each $v\in V(H_0)$.
Call $H$ a \emph{claw} of $C$ if $H_0$ is a vertex, an \emph{edge-claw} if $H_0$ is an edge, a \emph{path-claw} if $H_0$ is a path of length 2, and a \emph{pentagon-claw} if $H_0$ is a pentagon.

Let $\mathcal{G}$ denote all the connected plane graphs without cycles of length 4 or
6. For a cycle $C$, whose length is at most 11, of a graph from $\mathcal{G}$, $C$ is good if it contains no claws, edge-claws, path-claws or pentagon-claws; bad otherwise.

Let $G$ be a graph, $H$ a subgraph of $G$,
and $\phi$ a $(1,0,0)$-coloring of $H$. We say that $\phi$ can be
\emph{super-extended} to $G$ if $G$ has a $(1,0,0)$-coloring $c$ such $c(u)=\phi(u)$ for each $u\in V(H)$ and that $c(v)\neq c(w)$ whenever $v \in V(H)$, $w \in V(G) \setminus V(H)$ and
$vw\in E(G)$.

We shall prove the following theorem, called super-extended theorem, that is stronger than Theorem \ref{thm45tri7}.
\begin{theorem} (Super-extended theorem) \label{thm_main_extension}
Let $G\in \cal{G}$. If the boundary $D$ of the
unbounded face of $G$ is a good cycle, then every (1,0,0)-coloring
of $G[V(D)]$ can be super-extended to $G$.
\end{theorem}
By assuming the truth of Theorem \ref{thm_main_extension}, we can easily derive Theorem \ref{thm45tri7} as follows.
We may assume that $G$ is connected since otherwise, we argue on each component.
If $G$ has no triangles, then by Three Color Theorem, $G$ is 3-colorable.
Hence, we may assume that $G$ has a triangle, say $T$.
By Theorem \ref{thm_main_extension}, we can super-extend any given (1,0,0)-coloring of $T$ respectively to its interior and exterior.

The rest of this section contributes to some necessary notations.

Let $C$ be a cycle of a plane graph and $T$ be a claw, or an edge-claw, or a path-claw, or a pentagon-claw of $C$. We call the graph $H$ consisting of $C$ and $T$ a \textit{bad partition} of $C$. Every facial cycle (except $C$) of $H$ is called a \textit{cell} of $H$. 

The length of a path is the number of edges it contains. Denote by $|P|$ the length of a path $P$, by $|C|$ the length of a cycle $C$ and by $d(f)$ the size of a face $f$.  
A \textit{$k$-vertex} (resp. $k^+$-vertex and $k^-$-vertex) is a vertex $v$ with $d(v)=k$ (resp. $d(v)\geq k$ and $d(v)\leq k$). 
Similar notations are applied for paths, cycles and faces by constitute $d(v)$ for $|P|,|C|$ and $d(f)$, respectively.

Consider a plane graph. A vertex is \textit{external} if it lies on the exterior face; \textit{internal} otherwise.
A $3^+$-vertex is \emph{light} if it is internal and of degree 3; \emph{heavy} otherwise.
Let $d_1,d_2,d_3$ be three integers greater than 2. A \emph{$(d_1,d_2,d_3)$-face} is a 3-face whose vertices are all internal and have degree $d_1, d_2$ and $d_3$, respectively.
A $k$-cycle with vertices $v_1,\ldots,v_k$ in cyclic order is denoted by $[v_1\ldots v_k]$.
Let $f=[uxy]$ be a 3-face and $v$ be a neighbor of $u$ other than $x$ and $y$. If $u$ is an internal 3-vertex, then we call $v$ an \emph{outer neighbor} of $u$ (or of $f$),
$u$ a \emph{pendent vertex} of $v$, and $f$ a \emph{pendent 3-face} of $v$.
A 3-face is \emph{weak} if it has at least one outer neighbor that is light.
A path is a \textit{splitting path} of a cycle $C$ if its two end-vertices lie on $C$ and all other vertices lie inside $C$.
A cycle $C$ is \textit{separating} if neither $int(C)$ nor $ext(C)$ is empty.

\section{The proof of Theorem \ref{thm_main_extension}} \label{sec_proof}
Suppose to the contrary that Theorem \ref{thm_main_extension} is false.
From now on, let $G=(V,E)$ be a counterexample to Theorem \ref{thm_main_extension} with the smallest $|V|+|E|$.
Thus, we may assume that the boundary $D$ of the exterior face of $G$ is a good cycle, and that there exists a (1,0,0)-coloring $\phi$ of $G[V(D)]$ which cannot be super-extended to $G$.
By the minimality of $G$, we deduce that $D$ has no chord.

Denote by $\{1,2,3\}$ the color set for $\phi$ where the color 1 might be assigned to two adjacent vertices.
We define that, to \emph{3-color} a vertex $v$ means to assign $v$ with a color from $\{1,2,3\}$ when this color has not been used by its neighbors yet; and to \emph{(1,0,0)-color} $v$ means either to 3-color $v$ or to assign $v$ with the color 1 when precisely one neighbor of $v$ is of color 1.

\subsection{Structural properties of the minimal counterexample $G$}
\begin{lemma} \label{lem_min degree}
Every internal vertex of $G$ has degree at least 3.
\end{lemma}

\begin{proof}
 Suppose to the contrary that $G$ has an internal vertex $v$ of degree at most 2. We can super-extend $\phi$ to $G-v$ by the minimality of $G$, and then to $G$ by 3-coloring $v$.
\end{proof}

\begin{lemma} \label{lem_sep good cycle}
	$G$ has no separating good cycle.
\end{lemma}

\begin{proof}
	Suppose to the contrary that $G$ has a separating good cycle $C$. We super-extend $\phi$ to $G-int(C)$.
	Furthermore, since $C$ is a good cycle, the restriction of $\phi$ on $C$ can be super-extended to its interior, yielding a super-extension of $\phi$ to $G$.
\end{proof}

\begin{lemma}
$G$ is 2-connected. Particularly, the boundary of each face of $G$ is a cycle.
\end{lemma}

\begin{proof}
Otherwise, let $B$ a pendant block of $G$ of minimum order, and let $v$ be a cut vertex of $G$ associated with $B$.
By the minimality of $G$, we can super-extend $\phi$ to $G-(B-v)$. If we can 3-color $B$, then permute the color classes of $B$ so that the colors assigned to $v$ coincide, which completes a super-extension of $\phi$ to $G$. By the minimality of $B$, $B$ is 2-connected. If $B$ has no triangles, then Gr\"otsch's Theorem yields that $B$ is 3-colorable. So, let $T$ be a triangle of $B$. By Lemma \ref{lem_sep good cycle}, $T$ is a 3-face. Assign distinct colors to its three vertices, and by the minimality of $G$, we can super-extend the coloring of $T$, as an exterior face of $B$, to $B$. This gives a 3-coloring of $B$.
\end{proof}

By the definition of a bad cycle, one can easily conclude the following lemma.
\begin{lemma} \label{lem_bad-cycle}
If $C$ is a bad cycle of a plane graph of $\mathcal{G}$, then $C$ has a bad partition isomorphic to one of the eight graphs shown in Figure \ref{fig_claw}. In particular, $C$ has length 9 or 10 or 11. If $|C|=9$ then $C$ has a (5,5,5)-claw; if $|C|=10$ then $C$ has a (3,7,3,7)- or (5,5,5,5)-edge-claw, or a (5,5,5,5,5)-pentagon-claw; if $|C|=11$ then $C$ has a (3,7,7)- or (5,5,7)-claw, or a (3,7,3,8)-edge-claw, or a (5,5,5,5,5)-path-claw.
\end{lemma}

\begin{figure}[h]
	\centering
	\includegraphics[width=13cm]{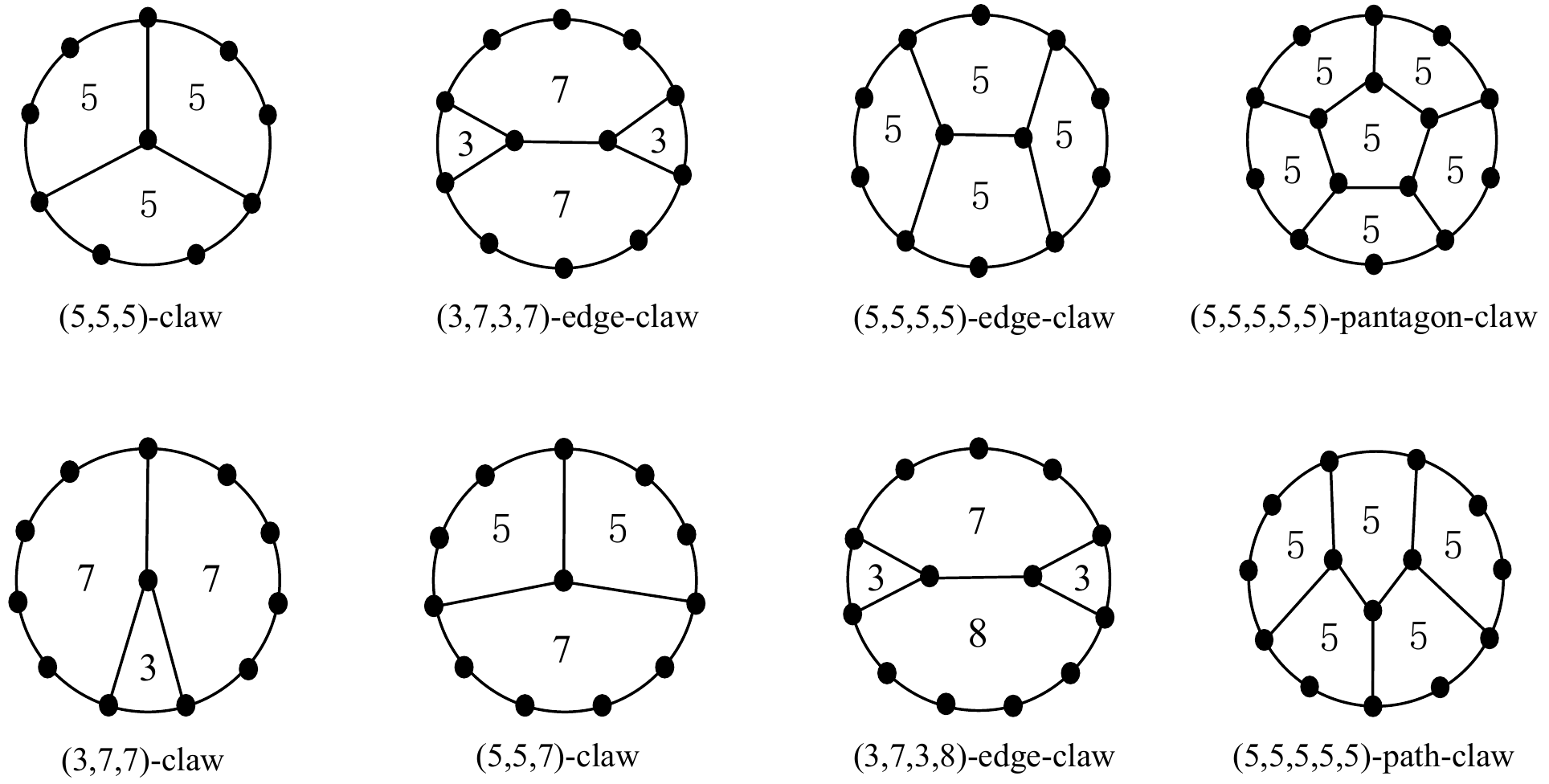}\\
	\caption{bad partitions of a cycle in a plane graph from $\mathcal{G}$, where the numbers indicate the length of each cell. A further name for the claw, edge-claw, path-claw or pantagon claw, which corresponds to each bad partition, is given below each drawing.}\label{fig_claw}
\end{figure}

From Lemmas \ref{lem_sep good cycle} and \ref{lem_bad-cycle}, one can deduce the following remark.  
\begin{remark} Let $C$ be a bad cycle of $G$. The following statements hold true.
\begin{enumerate} [(1)] \label{rem_bad_cycle}
	\setlength{\itemsep}{0pt}
	\item Every cell of $C$ is facial except that an 8-cell may have a (3,7)-chord connecting two vertices of $C$.
	\item Every vertex inside $C$ has degree 3 in $G$. \label{term_degree}
	\item Every vertex on $C$ has at most one neighbor inside $C$. \label{term_neighbor}
	\item Every vertex on $C$ is incident with at most two edges that locate inside $C$, where the exact case happens if and only if $C$ has a (3,7,3,8)-edge-claw. \label{term_incident-edges}
	\item For any set $S$ of four consecutive vertices on $C$, $G$ has at most two edges connecting a vertex from $S$ to a vertex inside $C$. \label{term_consecutive-four-vertices}
\end{enumerate}
\end{remark}

\begin{lemma}\label{lem_degree3}
	$G$ has no light vertex with neighbors all light. 
\end{lemma}
\begin{proof}
	Otherwise, let $v$ be such a light vertex. 
	Remove $v$ and its three neighbors, obtaining a smaller graph $G'$.
	By the minimality of $G$, $\phi$ can be super-extended to $G'$. 
	We further extend $\phi$ to being a (1,0,0)-coloring of $G$ in such way: 3-color all the neighbors of $v$ and consequently, $v$ can be (1,0,0)-colored.
\end{proof}

\begin{lemma}\label{lem_light_triangle}
	Every $(3,3,4)$-face of $G$ has no light outer neighbors.
\end{lemma}

\begin{proof}
	Suppose to the contrary that $f=[uvw]$ is a $(3,3,4)$-face of $G$
	having a light outer neighbor $x$. W.l.o.g.,
	Let $u$ be adjacent to $x$ and let $d(w)=4$.
    Remove $u,v,w$ and $x$ from $G$, obtaining a smaller graph $G'$. 
    By the minimality of $G$, $\phi$ can be
	super-extended to $G'$ and further to $G$ in such way: 3-color $w,v$ and $x$ in turn, and then (1,0,0)-color $u$.
\end{proof}

\begin{lemma}\label{lem_splitting path}
Let $P$ be a splitting path of $D$ which divides $D$ into two cycles, say $D'$ and $D''$.
The following four statements hold true.
\begin{enumerate}[(1)]
	\setlength{\itemsep}{0pt}
  \item If $|P|=2$, then there is a triangle between $D'$ and $D''$.
  \item If $|P|=3$, then there is a 5-cycle between $D'$ and $D''$.
  \item If $|P|=4$, then there is a 5- or 7-cycle between $D'$ and $D''$.
  \item If $|P|=5$, then there is a 7- or 8- or 9-cycle between $D'$ and $D''$.
\end{enumerate}
\end{lemma}

\begin{proof}
Since $D$ has length at most 11, we have $|D'|+|D''|=|D|+2|P|\leq 11+2|P|$.

(1) ~Let $P=xyz$. Suppose to the contrary that $|D'|, |D''| \geq 5$. It follows that $|D'|, |D''| \leq 10$.
By Lemma \ref{lem_min degree}, $y$ has a neighbor other than $x$ and $z$, say $y'$. The vertex $y'$ is internal since otherwise, $D$ is a bad cycle with a claw. W.l.o.g., let $y'$ lie inside $D'$. Now $D'$ is a separating cycle. By Lemma \ref{lem_sep good cycle}, $D'$ is not good. Recall that $|D'|\leq 10$. So $D'$ is a bad 9- or 10-cycle and $D''$ is a 5-cycle.
By Lemma \ref{lem_bad-cycle}, $D'$ has a (5,5,5)-claw  or a (5,5,5,5)-edge-claw or a (3,7,3,7)-edge-claw or a (5,5,5,5,5)-pentagon-claw, which would lead to a (5,5,5,5)-edge-claw or a (5,5,5,5,5)-path-claw of $D$ for the first two cases, to a 6-cycle for the third case, and to $y'$ being a light vertex with three light neighbors for the last case, a contradiction.

(2) Let $P=wxyz$. Suppose to the contrary that $|D'|, |D''| \geq 7$. It follows that $|D'|, |D''| \leq 10$.
Let $x'$ and $y'$ be neighbors of $x$ and $y$ not on $P$, respectively.
If both $x'$ and $y'$ are external, then $D$ has an edge-claw.
Hence, we may assume that $x'$ lies inside $D'$.
By Lemmas \ref{lem_sep good cycle} and \ref{lem_bad-cycle}, we deduce that $D'$ is a bad 9- or 10-cycle. So, $D''$ is a 7- or 8-cycle, which is good. Since every cell of $D'$ is facial, $y'$ must lie on $D''$. The application of this lemma to the splitting 2-path $y'yz$ yields that $yy'$ a (3,7)-chord of $D''$. So, $D'$ is a 9-cycle, which has a (5,5,5)-claw. Now the triangle $[yy'z]$ is adjacent to some 5-cell of $D'$, a contradiction.

(3) Let $P=vwxyz$. Suppose to the contrary that $|D'|, |D''| \geq 8$.
It follows that $|D'|, |D''| \leq 11$.
If $wy\in E(G)$, then by applying this lemma to the splitting 3-path $vwyz$ of $D$,
either $D'$ or $D''$ has length 6, a contradiction. Hence, $wy\notin E(G)$. Similarly, $vx,xz\notin E(G)$.
Since $G$ has no 4-cyles and $D$ has no chord, we can further conclude that $G$ has no edges connecting two nonconsecutive vertices on $P$, i.e., $G[V(P)]$ is $P$.

By Lemma \ref{lem_min degree}, $x$ has a neighbor $x'$ besides $w$ and $y$.
We claim that $x'$ lies inside $D$. Suppose to the contrary that $x'\in V(D')$.
By applying this lemma to the splitting 3-paths $vwxx'$ and $x'xyz$, $xx'$ is a (5,5)-chord of $D'$.
Since $d(w)\geq 3$, let $w'$ be a neighbor of $w$ other than $v$ and $x$.
Clearly, $w'$ lies either on $D''$ or inside it. Recall that $w'$ is not on $P$.
If $w'$ lies on $D''\setminus V(P)$, then $vww'$ is splitting 2-path of $D$, which forms a triangle adjacent to a 5-cell of $D'$, a contradiction. 
Hence, $w'$ lies inside $D''$. 
Similarly, $y'$ lies inside $D''$ as well. 
Clearly, $w'$ and $y'$ are distinct vertices. Notice that $w$ and $y$ have distance 2 along $D''$.
So, as a bad cycle, whose possible interior is given by Lemma \ref{lem_bad-cycle}, $D''$ has a (5,5,5,5)-edge-claw or a (5,5,5,5,5)-path-claw or a (5,5,5,5,5)-pentagon-claw, which implies a pentagon-claw of $D$ for the first case, and $w'$ being a light vertex with three light neighbors for the last two cases, a contradiction.

W.l.o.g., let $x'$ lies inside $D'$. So $D'$ is a bad cycle.
By Remark \ref{rem_bad_cycle}(\ref{term_degree}), $d(x')=3$.
Denote by $I$ the set of edges connecting a vertex from $\{w,x,y\}$ to a vertex not on $P$. 
Recall that $G[V(P)]$ is $P$. So, Lemma \ref{lem_min degree} implies that $|I|\geq 3$. 
By applying Lemma \ref{lem_degree3} to $x$, we further have $|I|\geq 4$.

Suppose that $D''$ is also a bad cycle, then one of $D'$ and $D''$ has length 9 and the other has length 9 or 10, which implies that one contains at most one edge from $I$ inside and the other contains at most two edges from $I$ inside, contradicting the fact that $|I|\geq 4.$
Hence, we may assume that $D''$ is a good cycle.

We conclude that $d(x)=3$. This is because $x$ has no neighbors on $D$ by the same argument as for $x'$, no neighbors inside $D''$ since $D''$ is a good cycle, and no neighbors besides $x'$ inside $D'$ by Remark \ref{rem_bad_cycle}(\ref{term_incident-edges}).

Recall that $D''$ is a good cycle, so $w$ (as well as $y$) has no neighbors inside $D''$. Moreover, since $D$ has no claws, $w$ (as well as $y$) has at most one neighbor on $D\setminus \{v,z\}$.
It follows with $|I|\geq 4$ that, inside $D'$ there exists a vertex $t$ adjacent to $w$ or $y$.
By Remark \ref{rem_bad_cycle}(\ref{term_neighbor}) and (\ref{term_consecutive-four-vertices}), such $t$ is unique. W.o.l.g, let $tw\in E(G)$. 
This implies that $|I|=4$ and each of $w$ and $y$ have a neighbor on $D-V(P)$.
If $t=x'$, then $[wxx']$ is a pendent $(3,3,4)$-face of $y$, contradicting Lemma \ref{lem_light_triangle}.
So, $t$ and $x'$ are distinct.
Moreover, $t$ and $x'$ are not adjacent since otherwise $G$ has a 4-cycle.
Hence, we can conclude that $D'$ has a path-claw or a pentagon-claw, making all cells of length 5.
This yields that $y$ mush have no neighbors other than $z$ on $D$, a contradiction.

(4) Let $P=uvwxyz$. Suppose to the contrary that $|D'|, |D''| \geq 10$.
Since $|D'|+|D''|\leq 21$, we have $|D'|, |D''| \leq 11$.
We claim that $G$ has no edges connecting two nonconsecutive vertices on $P$, i.e., $G[V(P)]$ is $P$.
Otherwise, let $e=t_1t_2$ be such an edge. 
Let $P'$ be obtained from $P$ by constituting $e$ for the subpath of $P$ between $t_1$ and $t_2$.
Clearly, $P'$ is a splitting $4^-$-path of $D$.
Applying this lemma to $P'$ yields that either $D'$ or $D''$ has length at most 8, a contradiction.
By this claim and Lemma \ref{lem_min degree}, we may let $v',w',x'$ and $y'$ be a neighbor of $v,w,x$ and $y$ not on $P$, respectively. 

We claim that both $w$ and $x$ have no neighbors on $D$.  Otherwise, w.l.o.g., let $w'$ be on $D'$.
By applying this lemma to the splitting 3-path $uvww'$ and the splitting 4-path $w'wxyz$ of $D$, we deduce that $ww'$ is a $(5,7)$-chord of $D'$.
Hence, the interior of $D'$ contains no edges incident with $v,x$ or $y$.
If $x'$ lies on $D''$ then similarly, $xx'$ is a $(5,7)$-chord of $D''$, resulting in no positions for $u'$ and $y'$, a contradiction.
Hence, $x'$ must lie inside $D''$. So, $D''$ is a bad cycle.
Since a bad cycle has at most one chord, Remark \ref{rem_bad_cycle}(\ref{term_consecutive-four-vertices}) implies that the interior of $D''$ contains at most three edges incident with $v,w,x$ or $y$.
It follows that $d(v)=d(w)=d(x)=d(y)=3$.
By Remark \ref{rem_bad_cycle}(\ref{term_degree}), $d(x')=3$.
Now $x$ is a light vertex with three light neighbors, contradicting Lemma \ref{lem_degree3}.

Suppose that one of $D'$ and $D''$, say $D'$, is a good cycle.
In this case, both $w'$ and $x'$ lie inside $D'$.  
Remark \ref{rem_bad_cycle}(\ref{term_neighbor}) implies that such $w'$ and $x'$ are unique.
So, $d(w)=d(x)=3$. 
By Remark \ref{rem_bad_cycle}(\ref{term_consecutive-four-vertices}), both $v'$ and $y'$ are on $D$. 
Clearly, such $v'$ and $y'$ are also unique since otherwise, $D$ has a claw.
So, $d(v)=d(y)=3$.
By Remark \ref{rem_bad_cycle}(\ref{term_degree}), $d(w')=d(x')=3$.
Now $x$ is a light vertex having three light neighbors, contradicting Lemma \ref{lem_degree3}.
Therefore, both $D'$ and $D''$ are bad.

Denote by $I$ the set of edges not on $P$ and incident with a vertex from $\{v,w,x,y\}$.
Notice that a bad cycle has a chord only if it is of length 11, but not both $D'$ and $D''$ have length 11.
So, $I$ has at most one edge taking a vertex on $D$ as an end. 
Moreover, Remark \ref{rem_bad_cycle}(\ref{term_consecutive-four-vertices}) implies that $I$ has at most four edges taking a vertex inside $D'$ or $D''$ as an end.
Therefore, $|I|\leq 5.$ This leads to the only case that $d(v)=d(y)=3$ and between $w$ and $x$, one has degree 3 and the other 4 since otherwise, at least one of $w$ and $x$ would be a light vertex with three light neighbors.
W.l.o.g., let $d(x)=4$. Since Remark \ref{rem_bad_cycle}(\ref{term_neighbor}), we may assume that $w'$ and $x'$ lie inside $D'$.
Lemma \ref{lem_light_triangle} implies that $w'$ and $x'$ can not coincide. Notice that $w$ and $x$ are consecutive on $D'$. By the specific interior of a bad cycle, we can deduce that $D'$ is a 11-cycle having a (5,5,5,5,5)-path-claw. This implies that both $D'$ and $D''$ have no chords, a contradiction.
\end{proof}

Loops and multiple edges are regarded as 1-cycles and 2-cycles, respectively.
\begin{lemma}\label{lem_general-operation}
	Let $G'$ be a connected plane graph obtained from $G$ by deleting vertices, inserting edges, identifying vertices, or any combination of them. If $G'$ is smaller than $G$ and the following holds:
	\begin{enumerate}[(i)] 
		\setlength{\itemsep}{0pt}
		\item identify no pair of vertices of $D$ and insert no edges connecting two vertices of $D$, and
		\item create no $k$-cycles for any $k\in\{1,2,4,6\}$, and
		\item $D$ is good in $G'$,
	\end{enumerate}
 then $\phi$ can be super-extended to $G'$.
\end{lemma}

\begin{proof}
	By Term $(ii)$, the graph $G'$ is simple and $G'\in \cal{G}$.
	The term $(i)$ guarantees that the new graph $G'$ has the same $D$ as the boundary of its exterior face, and that $\phi$ is a (1,0,0)-coloring of $G'[V(D)]$.
	Since $D$ is good in $G'$ and $G'$ is smaller than $G$, the lemma holds true by the minimality of $G$.
\end{proof}

\begin{lemma}\label{pro_operation}
Let $G'$ be a connected plane graph obtained from $G$ by deleting a set of internal vertices together with either identifying two vertices or inserting an edge between two vertices.
If the following holds true for this graph operation:
\begin{enumerate}[($a$)] 
\setlength{\itemsep}{0pt}
  \item identify no pair of vertices of $D$, insert no edges connecting two vertices of $D$, and
  \item create no $6^-$-cycles or triangular 7-cycles,
\end{enumerate}
then $\phi$ can be super-extended to $G'$.
\end{lemma}

\begin{proof}
Lemma \ref{lem_general-operation} shows that, to complete the proof, it suffices to showing that $D$ is a good cycle of $G'$.
Suppose to the contrary that $D$ has a bad partition $H$ in $G'$.
We distinguish two cases on the graph operation.

Case 1: assume that the graph operation includes identifying two vertices.
Denote by $v_1$ and $v_2$ the two vertices we identify and by $v$ the resulting vertex.
Lemma \ref{lem_bad-cycle} lists all the possible structure for $H$.
Recall that $D$ stays the same during the operation. If either $v\notin V(H)$ or $v\in V(H)$ such that $d_H(v)=2$, then $H$ stays the same during the operation, contradicting the fact that $D$ is a good cycle in $G$.
Hence, $v$ lies on $H$ and $d_H(v)=3$. If all the three neighbors of $v$ in $H$ are adjacent in $G$ to a common vertex from $\{v_1,v_2\}$, then again $H$ stays the same during the operation, a contradiction. Hence, one neighbor is adjacent to $v_1$ and the other two adjacent to $v_2$. This implies that there are two cells around $v$ that are created by our graph operation.
It follows by the possible structure of $H$ that, we create either a $6^-$-cycle or a triangular 7-cycle, contradicting the assumption (b).

Case 2: assume that the graph operation includes inserting an edge, say $e$.
Recall that $D$ stays the same during the operation.
If $e\notin E(H)\setminus E(D)$, then $H$ is a bad partition of $D$ also in $G$, a contradiction; otherwise, the two cells of $H$ containing $e$ are created by our operation, contradicting the assumption (b).
\end{proof}

\begin{lemma}\label{lem_2pendent}
	$G$ contains no internal 4-vertices having a pendent $(3,3,3)$-face and another pendent
	$(3,3,4^-)$-face.
\end{lemma}

\begin{proof}
Suppose to the contrary that $G$ has such a vertex $x$.
Denote by $[u_1u_2u_3]$ a $(3,3,3)$-face and by $[v_1v_2v_3]$ a $(3,3,4^-)$-face, with $u_1$ and $v_1$ as pendent vertices of $x$ and with $v_3$ as the $4^-$-vertex.
Denote by $x_1$ and $x_2$ the remaining neighbors of $x$. We distinguish two cases.

Case 1: assume that $x_1$ and $x_2$ lie on different sides of the path $u_1xv_1$, i.e., $x_1$ and $x_2$ are not consecutive in the cyclic order around $x$.
Remove $x,u_1,u_2,u_3,v_1,v_2$, $v_3$ from $G$ and identify $x_1$ with $x_2$, obtaining a
smaller graph $G'$ than $G$. If this operation satisfies both terms in Lemma \ref{pro_operation}, then the pre-coloring $\phi$ of $D$ can be super-extended to $G'$ by the minimality of $G$, and further to $G$ in such way:
3-color $v_3$, $v_2$, $v_1$, $x$, $u_2$, $u_3$ in turn and consequently, we can (1,0,0)-color $u_1$. 

(Term $a$) If our operation identifies two vertices of
$D$, or creates an edge that connects two vertices of $D$, then the path $x_1xx_2$ is contained in a splitting 2- or 3-path of $D$. By Lemma \ref{lem_splitting path}, this splitting path divides $D$ into two parts, one of which is a 3- or 5-cycle. So this cycle is a good cycle but now it separates $v_1$ from $u_1$, contradicting Lemma \ref{lem_sep good cycle}.

(Term $b$) If our operation creates a new $7^-$-cycle,
then this cycle corresponds to a $7^-$-path of $G$ between $x_1$ and $x_2$, which together with the path $x_1xx_2$ forms a $9^-$-cycle of $G$, say $C$.
Clearly, $C$ separates $u_1$ from $v_1$. So, $C$ is a bad
9-cycle having a $(5,5,5)$-claw. But now $C$ contains a 3-face inside, either $[u_1u_2u_3]$ or $[v_1v_2v_3]$, a contradiction.

Case 2: assume that $x_1$ and $x_2$ lie on the same side of the path $u_1xv_1$.
W.l.o.g., let $u_1,x_1,x_2,v_1$ locate in clockwise order around $x$ and so do $u_1,u_2,u_3$ along the cycle $[u_1u_2u_3]$.
Denote by $y$ the remaining neighbor of $u_2$. 
Delete $x$,$u_1$,$u_2$,$u_3$,$v_1$,$v_2$, $v_3$ and identify $x_2$ with $y$, obtaining a
smaller graph $G'$ than $G$. 
If our graph operation satisfies both terms of Lemma \ref{pro_operation}, then $\phi$ can be super-extended to $G'$ by the minimality of $G$ and further to $G$ in such way: 
3-color $x$ and $u_3$; since $x$ and $y$ receive different colors, we can 3-color $u_1$ and $u_2$; 3-color $v_3$ and $v_2$ in turn and finally, we can $(1,0,0)$-color $v_1$.

Let us show that both terms of Lemma \ref{pro_operation} do hold:

(Term $a$) Otherwise, the path $yu_2u_1xx_2$ is contained in a splitting 4- or 5-path of $D$. By Lemma \ref{lem_splitting path}, this splitting path divides $D$ into two parts, one of which is a $9^-$-cycle, say $C$. Now $C$ separates $v_1$ from $u_3$. Hence, $C$ is a bad 9-cycle with a (5,5,5)-claw.
But $C$ has to contain a 3-face inside, either $[u_1u_2u_3]$ or $[v_1v_2v_3]$, a contradiction.

(Term $b$) Suppose our operation creates a new $7^-$-cycle,
then it corresponds to a $7^-$-path of $G$ between $y$ and $x_2$, which together with the path $yu_2u_1xx_2$ forms a $11^-$-cycle of $G$, say $C$.
Clearly, $C$ separates $v_3$ from $u_3$. So $C$ is a bad cycle containing either $u_3$ or $v_3$ inside. For the former case, because of the existence of $[u_1u_2u_3]$ and $xx_1$, Remark \ref{rem_bad_cycle}(\ref{term_incident-edges}) implies that $xx_1$ is a chord of $C$, which thereby has a $(3,7,3,8)$-edge-claw. Now $u_3$ is a light vertex with three light neighbors, a contradiction to Lemma \ref{lem_degree3}. 
For the latter case, the interior of $C$, as a bad cycle, contains the triangle $[v_1v_2v_3]$, which is impossible.
\end{proof}

\begin{lemma}\label{lem_4vertex-1incident-1pendent}
$G$ contains no internal 4-vertice incident with a $(3,4^-,4)$-face and having a pendent $(3,3,4^-)$-face.	
\end{lemma}

\begin{proof}
Suppose to the contrary that such vertex exists, say $u$.
Denote by $u_1,\ldots, u_4$ the neighbors of $u$ locating in clockwise order around $u$. W.l.o.g., let $[uu_1u_2]$ be a $(3,4^-,4)$-face and $[u_3u_3'u_3'']$ be a pendent $(3,3,4^-)$-face of $u$.
Delete $u_1,u,u_3,u_3',u_3''$ from $G$ and identify $u_2$ with $u_4$, obtaining a new smaller graph $G'$. Similarly, to complete the proof, it suffices to doing two things.

Firstly, we shall show that both terms in Lemma \ref{pro_operation} hold.

(Term $a$)\quad If our operation identifies two vertices of
$D$, or creates an edge that connects two vertices of $D$, then the path $u_2uu_4$ is contained in a splitting 2- or 3-path of $D$. By Lemma \ref{lem_splitting path}, this splitting path divides $D$ into two parts, one of which is a 3- or 5-cycle, say $C$. Now $C$ separates $u_1$ from $u_3$, a contradiction.

(Term $b$)\quad If our operation creates a new 7$^-$-cycle, then $G$ has a $9^-$-cycle $C$ that contains the path $u_2uu_4$. Since $C$ separates $u_1$ from
$u_3$, $C$ is a bad 9-cycle with a (5,5,5)-claw, contradicting that $C$ contains a triangle either $[uu_1u_2]$ or $[u_3u_3'u_3'']$ inside.

Secondly, we shall show that any $(1,0,0)$-coloring of $G'$ can be super-extended to $G$. This can be done in the following way. Since one of $u_3'$ and $u_3''$ has degree 3 and the other degree at most 4, we can 3-color them. Notice that $u_1$ has degree either 3 or 4. Since $u_2$ and $u_4$ receive the same color,
if we can 3-color $u_1$, then consequently we can 3-color $u$ and (1,0,0)-color $u_3$ in turn, we are done.
Hence, we may assume that $u_1$ has degree 4 and its neighbors except $u$ are colored pairwise distinct. In this case, give the color of $u_2$ to $u_1$. Since $u_2$ has degree 3, we can recolor it properly. Since $u_1$ and $u_4$ are colored the same, we can 3-color $u$ and then (1,0,0)-color $u_3$. 
\end{proof}

\begin{lemma}\label{lem_4vertex-2incident}
	$G$ has no 4-vertices incident with two $(3,4^-,4)$-faces.
\end{lemma}

\begin{proof}
Suppose to the contrary that $G$ has such a 4-vertex $v$, incident with two $(3,4^-,4)$-faces $T_1=[vv_1v_2]$ and $T_2=[vv_3v_4]$. 
W.l.o.g., let $v_1,v_2,v_3,v_4$ locate in clockwise order around $v$. 

Case 1: assume that at least one of $T_1$ and $T_2$ is a $(3,3,4)$-face, w.l.o.g, say $T_1$.
Delete $v,v_1,\cdots,v_4$, obtaining a smaller graph $G'$ than $G$. Since we only remove vertices, both terms in Lemma \ref{pro_operation} hold. Hence, $\phi$ can be super-extended to $G'$ by the minimality of $G$, and further to $G$ in such way: 3-color the vertices of $T_2$. Denote by $v_1'$ and $v_2'$ the remaining neighbors of $v_1$ and $v_2$, respectively. We can always 3-color $v_1$ and $v_2$ except the case $\phi(v_1')=\phi(v_2')\neq \phi(v)$, for which we distinguish three subcases: 
if $1\notin \{\phi(v_1'),\phi(v)\}$, then give the color 1 to both $v_1$ and $v_2$, completing the super-extension;
if $\phi(v)=1$, then assign $v_1$ with the color 1 and consequently, we can 3-color $v_2$;
if $\phi(v_1')=1$, then recolor $v$ by the color 1, and then 3-color both $v_1$ and $v_2$.

Case 2: assume that both $T_1$ and $T_2$ are $(3,4,4)$-faces. W.l.o.g., let $d(v_1)=4$.
We distinguish two cases.

Case 2.1: assume that $d(v_3)=4$.
Denote by $v_2'$ and $v_4'$ the outer neighbors of $v_2$ and $v_4$, respectively. 
We delete all vertices of $T_1$ and $T_2$, and identify $v_2'$ with $v_4'$,
obtaining a new graph $G'$. 
We will show that both terms in Lemma \ref{pro_operation} do hold:

(Term $a$) 
If our operation identifies two vertices of
$D$, or creates an edge that connects two vertices of $D$, then the path $v_2'v_2vv_4v_4'$ is contained in a splitting 4- or 5-path of $D$. By Lemma \ref{lem_splitting path}, this splitting path divides $D$ into two parts, one of which is a $9^-$-cycle, say $C$. Now $C$ separates $v_1$ from $v_3$ and contains a triangle inside, a contradiction.

(Term $b$) 
If our operation creates a new 7$^-$-cycle, then $G$ has a $11^-$-cycle $C$ that contains the path $v_2'v_2vv_4v_4'$. Now $C$ separates $v_1$ from
$v_3$, both has degree 4, contradicting Remark \ref{rem_bad_cycle}(\ref{term_degree}).

We will show that any $(1,0,0)$-coloring of $G'$ can be super-extended to $G$:
3-color $v_1$ and $v_3$. Denote by $\alpha$ the color $v_2'$ and $v_4'$ received.
If $\alpha$ has not been used by both $v_1$ and $v_3$, then give $\alpha$ to $v$ and consequently, we can 3-color $v_2$ and $v_4$.
W.l.o.g., we may next assume that $v_3$ has color $\alpha$. 3-color $v_2$ and then (1,0,0)-color $v$. Since $v_3$ and $v_4'$ received the same color, we can 3-color $v_4$.

Case 2.2: assume that $d(v_4)=4$.
Denote by $v_i'$ the neighbor of $v_i$ for $i\in\{2,3\}$, 
and by $v_i'$ and $v_i''$ the remaining neighbors of $v_i$ locating in clockwise order around $v_i$ for $i\in\{1,4\}$？？？？？？？.
Delete all vertices of $T_1$ and $T_2$ and identify $v_1'$ with
$v_3'$. Denote by $z$ the resulting vertex and $G'$ the resulting graph. 
Notice that our operation may create some new $7^+$-cycles. 

Firstly, by the same argument as in Case 2.1, Term ($a$) does hold.

Secondly, we claim that the operation creates no $6^-$-cycles. Otherwise, $G$ has a $10^-$-cycle $C$ that contains the path $v_1'v_1vv_3v_3'$. So, $C$ is a bad cycle containing either $v_2$ or
$v_4$ inside. For the former case, since a bad $10^-$-cycle has no chords, $v_1$ has two neighbors inside $C$, contradicting Remark \ref{rem_bad_cycle}(\ref{term_neighbor}). For the latter case, $d(v_4)=4$ contradicts Remark 
\ref{rem_bad_cycle}(\ref{term_degree}).

Finally, we do not make $D$ bad. Otherwise, since we create no $6^-$-cycles, by the argument for the proof of Lemma \ref{pro_operation}, we can deduce that the new vertex $z$ is incident with two cells of $D$ in $G'$ that are created by our operation, where one cell has length 7 and the other length 7 or 8.
These two cells correspond to two cycles of $G$ containing the path $v_1'v_1vv_3v_3'$, one cycle (say $C'$) contains $v_2$ inside and the other (say $C''$) contains $v_4$ inside. 
Clearly, one of $C'$ and $C''$ has length 11 and the other length 11 or 12.
Since $d(v_4)=4$, we can deduce that $|C''|=12$ by Remark \ref{rem_bad_cycle} (2). So, $|C'|=11$. Hence, the way we make $D$ bad is that our operation make $D$ have a $(3,7,3,8)$-edge-claw in $G'$ where the 7-cell and 8-cell are created.
Let $e$ denote the common edge of these two cells.
Since $v_1$ is incident with two edges $v_1v_1''$ and $v_1v_2$ inside $C'$, we can deduce that $v_1v_1''$ is a chord of $C'$, which has a $(3,7,3,8)$-edge-claw in $G$ by Remark \ref{rem_bad_cycle} (3).
Let $C'=[v_3'v_3vv_1v_1'v_1''y_1\cdots y_5]$. 
Racall that $v_1'$ and $v_3'$ are the two vertices we identified. So, $e$ corresponds to either $v_3'y_5$ or $v_1'v_1''$.
For the former case, the vertices $v_1'',y_1,\cdots,y_4$ lie on $D$. A contradiction follows by applying Lemma \ref{lem_splitting path} to the splitting 4-path $v_1''v_1v_2v_2'y_4$ of $D$ in $G$.
For the latter case, by substituting $v_1v_1''$ for $v_1v_1'v_1''$ from $C''$, we obtain a 11-cycle of $G$ that contains $v_4$ inside, a contradiction.

Because of the conclusions in the previous three paragraphs, by the minimality of $G$, we can super-extend $\phi$ from $D$ to $G'$.
We complete a (1,0,0)-coloring of $G$ as follows:
3-color $v_4$ and $v_1$. Since $v_1$ and $v_3'$ receive different colors, we can 3-color $v_3$ and $v$. Finally, we can (1,0,0)-color $v_2$ except the case that $v_2'$ has the color 1 and between $v$ and $v_1$, one has the color 2 and the other 3.
Notice that the colors of $v_4$, $v_3$ and $v$ are pairwise distinct.
Recolor $v$ by 1 and finally, we can 3-color $v_2$.
\end{proof}

\begin{lemma}\label{lem_5vertex-2incident}
	$G$ has no internal 5-vertices incident with two faces, one is a weak $(3,3,5)$-face and the other is a
	 $(3, 4^-,5)$-face.
\end{lemma}

\begin{proof}
Suppose to the contrary that $G$ has such a vertex $v$. 
Denote by $v_1,\dots,v_5$ the neighbors of $v$ locating in clockwise order around $v$ with $[vv_1v_2]$ being a weak $(3,3,5)$-face and $[vv_3v_4]$ being a $(3, 4^-,5)$-face.
Let $x'$ be a light outer neighbor of $[vv_1v_2]$. Between $v_1$ and $v_2$, denote by $x$ the one adjacent to $x'$ and by $y$ the other. Clearly, $v_4$ is of degree 3 or 4. We distinguish two cases.

Case 1: assume $d(v_4)=3$.
Delete $v, v_1, v_2, x', v_4$ and identify $v_3$ with
$v_5$, obtaining a smaller graph $G'$ than $G$.
We shall show that both terms in Lemma \ref{pro_operation} hold.

(Term $a$) If our operation identifies two vertices of
$D$, or creates an edge that connects two vertices of $D$, then the path $v_3vv_5$ is contained in a splitting 2- or 3-path of $D$. By Lemma \ref{lem_splitting path}, this splitting path divides $D$ into two parts, one of which is a 3- or 5-cycle, say $C$. Now $C$ separates $v_2$ from $v_4$, a contradiction.

(Term $b$) If our operation creates a new 7$^-$-cycle, then $G$ has a $9^-$-cycle $C$ that contains the path $v_3vv_5$. Since $C$ separates $v_2$ from
$v_4$, $C$ is a bad 9-cycle with a (5,5,5)-claw, contradicting that $C$ contains a triangle either $[vv_1v_2]$ or $[vv_3v_4]$ inside.

Hence, the coloring $\phi$ of $D$ can be super-extended to $G'$ by Lemma \ref{pro_operation} and further to $G$ as follows: 3-color $v_4,v,x',y$ in turn and consequently, we can (1,0,0)-color $x$. This is a contradiction.

Case 2: assume $d(v_4)=4$. It follows that $d(v_3)=3$.
Let $v_3'$ be the remaining neighbor of $v_3$. 
Delete $v,v_1,v_2,v_3,v_4,x'$ and insert an edge
between $v_3'$ and $v_5$, obtaining a smaller graph $G'$ than $G$.

(Term $a$) Notice that our operation identifies no vertices. Suppose to the contrary that it creates an edge that connects two vertices of $D$, then the path $v_3'v_3vv_5$ is contained in a splitting 3-path of $D$. By Lemma \ref{lem_splitting path}, this splitting path divides $D$ into two parts, one of which is a $5$-cycle. Now this cycle separates $v_2$ from $v_4$, a contradiction.

(Term $b$) If our operation creates a new 7$^-$-cycle, then $G$
has a 9$^-$-cycle $C$ containing path $v_3'v_3vv_5$. Clearly, $C$
separates $v_2$ from $v_4$. Hence, $C$ is a bad 9-cycle that contains a triangle either $[vv_1v_2]$ or $[vv_3v_4]$ inside, a contradiction.

Hence, $\phi$ can be super-extended to $G'$ by Lemma \ref{pro_operation} and further to $G$ as follows: 3-color $v_4$. If $\phi(v_3')\neq \phi(v_5)$ or $\phi(v_3')=\phi(v_5)=\phi(v_4)$, then we can first 3-color $v$ and $v_3$, next 3-color $x'$ and $y$ in turn and consequently, we can (1,0,0)-color $x$, we are done. Hence, we may assume that $\phi(v_3')=\phi(v_5)\neq \phi(v_4)$. Since $v_3'$ and $v_5$ are adjacent in $G'$, both $v_3'$ and $v_5$ have color 1 and have no other neighbors colored 1.  So we can give the color 1 to $v_3$ and then 3-color $v$. By the same way as above, we color $v_2,v_1$ and $v$, we are done as well.
\end{proof}

\begin{lemma}\label{lem_5vertex-2incident-1pendent}
	If $v$ is an internal 5-vertex of $G$ incident with two 3-faces, one is a weak $(3,3,5)$-face and the other is a weak $(3,5,5^+)$-face,
	then $v$ has no pendent $(3,3,3)$-faces.
\end{lemma}

\begin{proof}
	Denote by $v_1,\ldots,v_5$ the neighbors of $v$, whose order around $v$ has not been given yet.
	Suppose to the contrary that $v$ has a pendent $(3,3,3)$-face, say $[v_1w_1w_2]$.
	Let $[vv_2v_3]$ be a weak $(3,3,5)$-face with $v_3'$ being a light outer neighbor of $v_3$.
	Let $[vv_4v_5]$ be a weak $(3,5,5^+)$-face with $v_4'$ being a light outer neighbor of $v_4$.
    Delete $v,v_1,\ldots,v_4,w_1,w_2,v_3',v_4'$ from $G$, obtaining a
	graph $G'$. By the minimality of $G$, the pre-coloring $\phi$ of $D$ can be
	super-extended to $G'$, and further to $G$ in such way: 3-color $v_4',
	v_4$ and $v$ in turn. If $v$ has color 1, then exchange the colors of $v$ and $v_4$. Hence, w.l.o.g., we may assume that $v$ has color 2. 3-color $v_3',v_2,w_1,w_2$ in turn. Consequently, we can $(1,0,0)$-color $v_3$ and $v_1$.
\end{proof}

\begin{lemma} \label{lem_6vertiex}
	If $v$ is an internal 6-vertex of $G$ incident with two weak $(3,3,6)$-faces,
	then $v$ is incident with no other $(3,4^-,6)$-faces,
\end{lemma}

\begin{proof}
Denote by $v_1,\ldots,v_6$ the neighbors of $v$ locating around $v$ in clockwise order.
Let $[vv_3v_4]$ and $[vv_5v_6]$ be two weak $(3,3,6)$-faces. 
Suppose to the contrary that $[vv_1v_2]$ is a $(3,4^-,6)$-face.
W.l.o.g., let $d(v_2)=3$.
Denote by $v_i'$ the remaining neighbor of $v_i$ for $i\in\{2,\dots,6\}$. Since $[vv_3v_4]$ is weak, denote by $x'$ a light outer neighbor of $[vv_3v_4]$. Between $v_3$ and $v_4$, denote by $x$ the one adjacent to $x'$ and by $y$ the other.
Delete vertices $v,v_1,\ldots,v_6,x'$ from $G$ and identify $v_2'$ with $v_5'$, obtaining a new graph $G'$. 
We will show that both terms in Lemma \ref{pro_operation} do hold:

(Term $a$) Otherwise, the path $v_2'v_2vv_5v_5'$ is contained in a splitting 4- or 5-path of $D$. By Lemma \ref{lem_splitting path}, this splitting path divides $D$ into two parts, one of which is a $9^-$-cycle, say $C$. Now $C$ separates $v_4$ from $v_6$ and contains a triangle either $[vv_3v_4]$ or $[vv_5v_6]$ inside, a contradiction.

(Term $b$) If our operation creates a new 7$^-$-cycle, then $G$ has a $11^-$-cycle $C$ that contains the path $v_2'v_2vv_5v_5'$. Since $C$ separates $v_4$ from
$v_5$, $C$ is a bad cycle. Now $v$ is a vertex on $C$ which has two neighbors either $v_3,v_4$ or $v_1,v_6$ inside $C$, contradicting Remark \ref{rem_bad_cycle}(\ref{term_neighbor}).

By Lemma \ref{pro_operation}, $\phi$ can be super-extended to $G'$. We will further super-extend $\phi$ to $G$ in the following way.
Let $\alpha$ be the color $v_2'$ and $v_5'$ receive.
3-color $v_1,v_2$ and $v$ in turn. If $v$ has color $\alpha$, then we can 3-color $v_6$ and $v_5$ in turn and seperately, 3-color $x'$ and $y$ in turn and then (1,0,0)-color $x$, we are done. Hence, we may assume that the color of $v$ is not $\alpha$. Since the colors of $v,v_1$ and $v_2$ are pairwise distinct, $v_1$ has color $\alpha$. We may assume that the color of $v$ is not 1 since otherwise, we exchange the colors of $v$ and $v_2$.
3-color $x'$ and $y$ in turn and consequently, we can (1,0,0)-color $x$. Remove the color of an outer neighbor (say $z$) of $[vv_5v_6]$ and in the same way, we color $z,v_5,v_6$, as desired.
\end{proof}

Let $W$ be a subgraph of $G$ consisting of a $(4,4,4)$-face $[uvw]$ and three 3-faces
$[uu_1u_2],[vv_1v_2]$ and $[ww_1w_2]$ of $G$ that share precisely one vertex (respectively, $u,v$ and $w$) with $[uvw]$. Let $u,v,w$ as well as $u_1,u_2,v_1,v_2,w_1,w_2$ be in clockwise order around $[uvw]$. 
Call $W$ a \emph{wheel}, written as $(uvw,u_1u_2v_1v_2w_1w_2)^\mathcal{W}$, if $d(u_1)=d(v_1)=d(w_1)=3$ and $d(u_2)=d(v_2)=d(w_2)=4$.
Call $W$ an \emph{antiwheel}, written as $(uvw,u_1u_2v_1v_2w_1w_2)^\mathcal{AW}$, if $d(u_1)=d(v_1)=d(w_2)=3$ and $d(u_2)=d(v_2)=d(w_1)=4$.

\begin{lemma}\label{lem_wheel}
	$G$ has no wheels.
\end{lemma}

\begin{proof}
Suppose to the contrary that $G$ has a wheel, say $W=(uvw,u_1u_2v_1v_2w_1w_2)^\mathcal{W}$. 
Let $u_1',v_1'$ and $w_1'$ be the remaining neighbors of $u_1,v_1$ and $w_1$, respectively.
Delete all vertices of $W$ and insert three edges making $[u_1'v_1'w_1']$ a triangle. We thereby obtain a graph $G'$ smaller than $G$. We shall use Lemma \ref{lem_general-operation}.

Suppose that our operation connects two vertices of $D$. 
W.l.o.g., let $u_1'$ and $v_1'$ locate on $D$. 
Then as a splitting 5-path of $D$, $u_1'u_1uvv_1v_1'$ divides $D$ into two parts, one of which is a $9^-$-cycle. Now this cycle separates $u_2$ from $w$ and contains a triangle either $[uu_1u_2]$ or $[uvw]$ inside, a contradiction.
Hence, Term ($i$) holds true. 

Suppose that our operation creates a new $7^-$-cycle $C'$ other than $[u_1'v_1'w_1']$.
Since $C'$ is new, $C'$ must share edges with $[u_1'v_1'w_1']$. If they have precisely two common edges (w.l.o.g., say $u_1'v_1'$ and $v_1'w_1'$), then the cycle obtained from $C'$ by constituting the edge $u_1'w_1'$ for the path $u_1'v_1'w_1'$ is also created and has smaller length than $C'$.
Take this cycle as the choice for $C'$. Hence, we may assume that $C'$ and $[u_1'v_1'w_1']$ have one edge in common, say $u_1'v_1'$. 
So, $C'$ corresponds to a $11^-$-cycle $C$ of $G$ that contains the path $u_1'u_1uvv_1v_1'$. Since $C$ separates $u_2$ from
$w$, $C$ is a bad cycle containing either $u_2$ or $w$ inside, both of which have degree 4. This contradicts Remark \ref{rem_bad_cycle} (\ref{term_degree}). Therefore, our operation creates no $7^-$-cycles $C'$ other than $[u_1'v_1'w_1']$. In particular, Term ($ii$) holds true.

Suppose that our operation makes $D$ bad. So, $D$ has a bad partition $H$ in $G'$. If $H$ and $[u_1'v_1'w_1']$ have no edges in common, then $H$ is a bad partition of $D$ in $G$ as well, a contradiction. Hence, let $e$ be a common edge of $H$ and $[u_1'v_1'w_1']$. Recall that among the vertices of $[u_1'v_1'w_1']$, at most one lies on $D$. So, $e$ is not an edge of $D$. This implies that $e$ is incident with two cells of $H$, both of which are new. That is to say, we created a $7^-$-cycle other than $[u_1'v_1'w_1']$, a contradiction. Therefore, Term ($iii$) holds true.

By Lemma \ref{lem_general-operation}, $\phi$ can be super-extended to $G'$. We will further super-extend $\phi$ to $G$. Since $[u_1'v_1'w_1']$ is a triangle of $G'$, we distinguish two cases as follows.

Case 1: assume that the colors of $u_1',v_1'$ and $w_1'$ are pairwise distinct.
 W.l.o.g., let $\phi(u_1')=3,\phi(v_1')=2$ and $\phi(w_1')=1$.
3-color $u_2,v_2$ and $w_2$. If $\phi(u_2)\neq 3$ and $\phi(v_2)\neq 2$, then assign $u,v,w$ with colors $3,2,1$, respectively. Consequently, we can 3-color $u_1,v_1$ and $w_1$, we are done.
W.l.o.g., we may next assume that $\phi(u_2)=3$. Assign $u_1$ with color $2$ and $u$ with color $1$. Since $u$ and $v_1'$ have different colors, we can 3-color $v$ and $v_1$. If $w_2$ has color 1, then we can 3-color $w$ and $w_1$ in turn; otherwise, assign $w$ with the color 1 and then $3$-color $w_1$.

Case 2: assume that the colors of $u_1',v_1'$ and $w_1'$ are not pairwise distinct.
Since the extension of $\phi$ in $G'$ is a $(1,0,0)$-coloring, precisely two of $u_1',v_1'$ and $w_1'$ have the color $1$, say $u_1'$ and $v_1'$. 
3-color $u_2,v_2,w_2,w_1,w$ in turn. We may assume that the color of $w$ is not 1 since otherwise, we can exchange the colors of $w$ and $w_1$.
W.l.o.g., let $w$ be of color 3.
Since both $u_1'$ and $v_1'$ have color 1 that is different from the color of $w$,
regardless of the edge $uv$, we can 3-color $u,u_1$ and $v,v_1$. 
The resulting coloring gives a (1,0,0)-coloring of $G$ unless both $u$ and $v$ have color 2.
For this remaining case, we can deduce that $u_1$ has color $3$ and $u_2$ has color 1.
Reassign $u$ with the color 1, we are done.
\end{proof}

\begin{lemma}\label{lem_antiwheel}
	$G$ has no antiwheel whose outer neighbors are all light.
\end{lemma}

\begin{proof}
	Suppose to the contrary that $G$ has such an antiwheel, say $W=(uvw,u_1u_2v_1v_2w_1w_2)^\mathcal{AW}$.
	Denote by $u_1',v_1'$ and $w_2'$ outer neighbors of $u_1,v_1$ and $w_2$, respectively.
	Delete all the vertices of $W$ except $v_2$, identify $v_2$ with $w_2'$, and insert an edge between $u_1'$ and $v_1'$, obtaining a new graph $G'$ from $G$. We shall use Lemma \ref{lem_general-operation}.
	
	Suppose that our operation identifies two vertices of
	$D$, or inserts an edge that connects two vertices of $D$. 
	So, $D$ has a splitting 4- or 5-path in $G$ containing either $v_2vww_2w_2'$ or $u_1'u_1uvv_1v_1'$. By Lemma \ref{lem_splitting path}, this splitting path divides $D$ into two parts, one of which is a $9^-$-cycle, say $C$. Now $C$ separates $u_2$ from $w_1$ and contains a triangle either $[uu_1u_2]$ or $[ww_1w_2]$ inside, a contradiction.
	Hence, Term ($i$) holds true. 
	
	Suppose that our operation creates a new $7^-$-cycle, say $C'$.
    $C'$ corresponds to a subgraph (say $P$) of $G$ that can be distinguished in four cases: (1) a $6^-$-path between $u_1'$ and $v_1'$; (2) a $7^-$-path between $w_2'$ and $v_2$; (3) the union of two vertex-disjoint paths, one between $u_1'$ and $w_2'$ and the other between $v_1'$ and $v_2$; (4) the union of two vertex-disjoint paths, one between $u_1'$ and $v_2$ and the other between $v_1'$ and $w_2'$. For the first case, $P$ and the path $u_1'u_1uvv_1v_1'$ together form a $11^-$-cycle which contains a 4-vertex either $u_2$ or $w_1$ inside, a contradiction to Remark \ref{rem_bad_cycle} (\ref{term_degree}). For the case (2), $P$ and the path $w_2'w_2wvv_2$ together form a $11^-$-cycle which contains a 4-vertex either $u_2$ or $w_1$ inside, again a contradiction to Remark \ref{rem_bad_cycle} (\ref{term_degree}).
	For the case (3), since $G$ has no $6^-$-cycles adjacent to a triangle, we can deduce that $G$ has no $4^-$-paths between $v_1'$ and $v_2$ by the existence of $[vv_1v_2]$ and no edges between $u_1'$ and $w_2'$ by the existence of $[uvw]$. 
	It follows that $P$ has length at least 8, a contradiction. Case (4) is impossible by the planarity of $G$.
	Therefore, our operation creates no $7^-$-cycles. In particular, Term ($ii$) holds true.

	Suppose that our operation makes $D$ bad. Let $H$ be a bad partition of $D$ in $G'$.
	Since both terms of Lemma \ref{pro_operation} holds, if $u_1'v_1'\notin E(H)$, then	
	the proof of Lemma \ref{pro_operation} shows that identifying $w_2'$ with $v_2$ can not make $D$ bad.
	So, $u_1'v_1'$ belongs to $H$. Since Term $(i)$ holds true, $u_1'v_1'$ is incident with two cells of $H$. Clearly, these two cells are created and at least one of them is a $7^-$-cycle, contradicting the conclusion above that our operation creates no $7^-$-cycles.
	Therefore, Term ($iii$) holds.
	
	By Lemma \ref{lem_general-operation}, $\phi$ can be super-extended to $G'$. 
	Denote by $\alpha$ the color $v_2$ and $w_2'$ receive and by $\beta$ the color $u_1'$ receives.
	3-color $u_2$ and $w_1$. We distinguish two cases according to the colors of $u_2$ and $w_1$.
	
	Case 1: suppose that not both $u_2$ and $w_1$ have color $\alpha$. So, we can 3-color $u,v$ and $w$. Since both $u_1'$ and $w_2'$ have degree 3,  we can 3-recolor them. Consequently, we can $(1,0,0)$-color $u_1$ and $w_2$. If not all the colors occur on the neighbors of $v_2$, then we can 3-recolor $v_2$ and eventually, 3-recolor $v_1'$ and $(1,0,0)$-color $v_1$ in turn, we are done. So, we may next assume that $v_2$ has all the colors around. It follows that $v_2$ is of color 1 and $v$ not. W.l.o.g., Let $v$ be of color 3. We may assume that $v_1'$ is of color 2 since otherwise, we can 3-color $v_1$. Since $G'$ has an edge between $u_1'$ and $v_1'$,  $\beta\neq 2$. Now we recolor some vertices as follows. Assign $v_1$ with 1, reassign $v_2$ with 3 and $v$ with 2, remove the colors of $u_1,u,w,w_2$, and give the color 1 back to $w_2'$ and $\beta$ back to $u_1'$. 
	 Since now $u_1'$ and $v$ have different colors, we can 3-color $u$ and $u_1$. 
	Clearly, $w_2'$ has no neighbors of color 2 since $v_2$ already has one. 
	If $w_1$ has color 2, then we can 3-color $w$ and (1,0,0)-color $w_2$ in turn;
	otherwise, assign $w_2$ with 2 and we can (1,0,0)-color $w$.

	Case 2: suppose that both $u_2$ and $w_1$ have color $\alpha$.
	If $\alpha=1$, then assign $u$ with $\alpha$ and we can 3-color $u_1,v_1,v,w,w_2$ in turn, we are done. W.l.o.g., we may next assume that $\alpha=2$. If $\beta \neq 3$, then we can 3-color $v_1,v,w,w_2$ in turn, assign $u$ with the color 1, and 3-color $u_1$ at last; otherwise, since $v_1'$ is of color different from $\beta$, we assign $u, w_2$ and $v_1$ with 3, and $u_1,w$ and $v$ with 1. We are done in both situations. 
\end{proof}

\begin{lemma}\label{lem_5-face-all-light}
	$G$ has no 5-faces whose vertices are all light.
\end{lemma}

\begin{proof}
Suppose $G$ has such a 5-face, say $f=[u_1u_2\ldots u_5]$.
For $i\in\{1,2,\ldots,5\}$, let $u_i'$ denote the remaining neighbor of $u_i$.
If both $u_1'$ and $u_3'$ belong to $D$, then as being a splitting 4-path of $D$, $u_1'u_1u_2u_3u_3'$ divides $D$ into two parts, one of which is a 5- or 7-cycle. This cycle is actually a face but now contains an edge either $u_2u_2'$ or $u_3u_4$ inside, a contradiction.
Therefore, at least one of $u_1'$ and $u_3'$ is internal.
For the same reason, this is even true for $u_i$ and $u_{i+2}$ for each $i\in \{1,2,\ldots,5\}$, where the index is added in modulo 5.
Hence, we can alway get three internal vertices $u_i'$, $u_{i+1}'$ and $u_{i+2}'$ for some $i\in \{1,2,\ldots,5\}$.
W.l.o.g., let $u_5'$, $u_1'$ and $u_2'$ be internal.
Remove all the vertices of $f$ from $G$ and insert an edge between $u_2'$ and $u_5'$, obtaining a new graph $G'$.
We shall use Lemma \ref{lem_general-operation}.
Clearly, Term $(i)$ holds true.

Suppose the graph operation creates a $k$-cycle with $k\in \{1,2,4,6\}$. So, $G$ has a $k$-path between $u_2'$ and $u_5'$. This path together with $u_5'u_5u_1u_2u_2'$ form a $(k+3)$-cycle, say $C$. 
By Lemma \ref{lem_degree3}, $d(u_1')\geq 4$. So, $C$ can not contain $u_1'$ inside since otherwise, a contradiction to Remark \ref{rem_bad_cycle} (\ref{term_degree}). Moreover, as a $9^-$-cycle, $C$ can not contain both $u_3$ and $u_4$ inside. Therefore, by planarity of $G$, $u_1'$ must locate on $C$.
Now the cycle, obtained from $C$ by constituting $u_2u_3u_4u_5$ for $u_2u_1u_5$, is a bad 10-cycle but it has a claw and a 5-cell, which is impossible. Therefore, Term $(ii)$ holds true.

Suppose that our operation makes $D$ bad. Let $H$ be a bad partition of $D$ in $G'$. So, $u_2'u_5'$ belongs to $H-E(D)$  since otherwise, $H$ is a bad partition of $D$ in $G$.
Now, $u_2'u_5'$ is incident with two cells of $H$, say $h'$ and $h''$.
Denote by $C'$ and $C''$ cycles obtained from $h'$ and $h''$ by constituting the edge $u_2'u_5'$ for the path $u_2'u_2u_1u_5u_5'$. 
Clearly, one of $C'$ and $C''$ (w.l.o.g., say $C'$) contains $u_1'$ inside or on $C$, and the other contains $u_3'$ and $u_4'$ inside. Since a cell has length at most 8, both $C'$ and $C''$ have length at most 11.
So, $C''$ is a bad cycle. Lemma \ref{lem_degree3} implies that both $u_3'$ and $u_4'$ are not light. So, $C''$ can not contains them inside by Remark \ref{rem_bad_cycle}(\ref{term_degree}). Instead, $u_3'$ and $u_4'$ are on $C''$. 
Now $C''$ has an edge-claw, more precisely, an (5,5,5,5)-edge-claw. So, $h''$ is a non-triangular 7-cell of $H$, which implies that $H$ must have a $(5,5,7)$-claw in $G'$. 
This gives a contradiction since both $u_2'$ and $u_5'$ are internal vertices on $H$.
Therefore, Term ($iii$) holds true.

By Lemma \ref{lem_general-operation}, $\phi$ can be super-extended to $G'$ and further to $G$ as follows.
If there is a vertex from $\{u_2',u_5'\}$ of color different from 1, w.l.o.g., say $u_5'$, then 3-color $u_1,\dots,u_4$ in turn and finally, we can $(1,0,0)$-color $u_5$.
So we may assume that both $u_2'$ and $u_5'$ are of color 1. Again, 3-color $u_1,\dots,u_4$ in turn. Since $u_2'$ has no neighbors of color 1 in $G$, we can $(1,0,0)$-color $u_5$. 
\end{proof}

\begin{lemma}\label{lem_5-face-4-light}
	$G$ has no 5-faces, four of whose vertices are light and the remaining one is an internal 4-vertex.  
\end{lemma}

\begin{proof}
	Suppose to the contrary the $G$ has such a 5-face $[u_1\ldots u_5]$.
	W.l.o.g., let $u_1$ be of degree 4. Denote by $u_1'$ and $u_1''$ the remaining neighbor of $u_1$ and for $i\in\{2,\dots,5\}$, denote by $u_i'$ the remaining neighbor of $u_i$.
	Remove all the vertices of $[u_1\ldots u_5]$ and insert an edge between $u_2'$ and $u_5'$, obtaining a new graph $G'$.
	We will show that both terms in Lemma \ref{pro_operation} do hold:
	
	(Term $a$) Otherwise, both $u_2'$ and $u_5'$ belong to $D$. 
	So, $u_2'u_2u_1u_5u_5'$ is a splitting 4-path of $D$, which divides $D$ into two parts so that one part is a 5- or 7-cycle $C$, by Lemma \ref{lem_splitting path}.  
	Notice that $C$ is actually a face but now has to contain an edge either $u_1u_1'$ or $u_2u_3$ inside, a contradiction.
	
	(Term $b$) Otherwise, $G$ has a $9^-$-cycle or a triangular 10-cycle $C$ containing the path $u_2'u_2u_1u_5u_5'$. By the planarity of $G$, either $C$ contains the edges $u_1u_1'$ and $u_1u_1''$ inside or $C$ contains the vertices $u_3$ and $u_4$ inside.
	For the former case, since $C$ has length at most 10, Remark \ref{rem_bad_cycle}(\ref{term_incident-edges}) implies that $C$ is not a bad cycle. So, $u_1'$ and $u_1''$ locate on $C$, yields the length of $C$ at least 11, a contradiction.
	For the latter case, by Lemma \ref{lem_degree3}, neither $u_3'$ nor $u_4'$ is light. So, they both locate on $C''$, implied by Remark \ref{rem_bad_cycle}(\ref{term_degree}). Now $C$ has a $(5,5,5,5)$-edge-claw, which gives a new triangular 7-cycle in $G'$, a contradiction.
		
	By Lemma \ref{pro_operation}, $\phi$ can be super-extended to $G'$ and further to $G$ in the same way as in the proof of Lemma \ref{lem_5-face-all-light}.
\end{proof}

\begin{lemma}\label{lem_two-5-faces-adjacent}
	$G$ has no two 5-faces $f$ and $g$ sharing precisely one edge, say $uv$, such that $u$ is an internal 5-vertex and all other vertices on $f$ or $g$ are light.
\end{lemma}

\begin{proof}
	Suppose to the contrary that such $f$ and $g$ exist. By the minimality of $G$, we can super-extend $\phi$ to $G-V(f)\cup V(g)$ and further to $G$ as follows: 3-color the vertices of $f$ and $g$ except $v$ beginning with $u$ along seperately the boundary of $f$ and one of $g$. Eventually, we can (1,0,0)-color $v$.
\end{proof}

\subsection{Discharging in $G$}\label{secch}
Let $u$ be a vertex of a $(4,4,4)$-face. $u$ is \emph{abnormal} if it is incident with a $(3,4,4)$-face; otherwise, $u$ is \emph{normal}. A 5-face is \emph{small} if it contains precisely four light vertices.
Let $P$ be the common part of $D$ and a face $f$. $f$ is \emph{sticking} if $P$ is a vertex, {$i$-ceiling} if $P$ is a path of length $i$ for $i\geq 1.$

Let $V, E$ and $F$ be the set of vertices, edges and faces of $G$, respectively.
Denote by $f_0$ the exterior face of $G$.
Give \emph{initial charge} $ch(x)$ to each element $x$ of $V\cup F$ defined as $ch(f_0)=d(f_0)+24$, $ch(x)=5d(x)-14$ for $x\in V$, and $ch(x)=2d(x)-14$ for $x\in F\setminus \{f_0\}$.
Move charges among elements of $V\cup F$ based on the following rules (called discharging rules):
\begin{enumerate}[$R1.$]
  \setlength{\itemsep}{0pt}
  \item Every internal 3-vertex sends to each incident face $f$  charge 1 if $d(f)=3$, and charge $\frac{1}{3}$ otherwise. \label{term_rule_3-vertex}
  \item Every internal 4-vertex sends to each incident 3-face $f$ charge $\frac{7}{2}$ if $f$ is a $(3,4,4)$-face, charge $3$ if $f$ is a $(3,3,4)$-face, charge $\frac{8}{3}$ if $f$ is a $(4,4,4)$-face, charge $\frac{5}{2}$ otherwise. \label{term_rule_4-vertex_incident}
  \item Every internal $5$-vertex sends to each incident 3-face $f$ charge $6$ if $f$ is weak $(3,3,5)$-face, charge $\frac{9}{2}$ if $f$ is $(3,4,5)$-face, charge $\frac{7}{2}$ if $f$ is either a weak $(3,5,5)$-face or a strong $(3,3,5)$-face, charge $3$ otherwise. \label{term_rule_5^+-vertex_incident}
  \item Every internal $6$-vertex sends to each incident 3-face $f$ charge $6$ if $f$ is weak $(3,3,6)$-face, charge $5$ if $f$ is $(3,4,6)$-face, charge $4$ otherwise. \label{term_rule_6-vertex_incident}
  \item Every internal $7^+$-vertex sends to each incident 3-face charge $6$. \label{term_rule_7-vertex_incident}
  \item Every internal $4^+$-vertex sends to each pendent 3-face $f$ charge $\frac{5}{3}$ if $f$ is $(3,3,3)$-face, charge $\frac{3}{2}$ if $f$ is a $(3,3,4)$-face, and charge $\frac{5}{4}$ otherwise.\label{term_rule_heavy_pendent}
  \item Every internal $4^+$-vertex $u$ sends to each incident 5-face $f$ charge $\frac{8}{3}$ if $d(u)\geq 5$ and $f$ is small, and charge $\frac{3}{2}$ otherwise.\label{term_rule_5-face}
  \item Within a $(4,4,4)$-face, every normal vertex send to each abnormal vertex charge $\frac{1}{6}$. \label{term_rule_normal-2-abnormal}
  \item Within an antiwheel, every strong $(3,4,4)$-face sends to each vertex of the $(4,4,4)$-face charge $\frac{1}{6}$. \label{term_rule_face-2-abnormal}
  \item The exterior face $f_0$ sends charge $3$ to each incident vertex.\label{term_rule_exterior_face}
  \item Every 2-vertex receives charge 1 from its incident face other than $f_0$. \label{term_rule_2-vertex}
  \item Every exterior $3^+$-vertex sends to each sticking 3-face charge $6$, to each ceiling 3-face charge $\frac{7}{2}$, to each sticking 5-face charge $\frac{8}{3}$, to each 2-ceiling 5-face charge $\frac{13}{6}$, to each pendent 3-faces charge $\frac{5}{3}$, to each 1-ceiling 5-face charge $\frac{3}{2}$, to each 3-ceiling 7-face charge 1, to each 2-ceiling 7-face charge $\frac{1}{2}$. \label{term_rule_exterior-3plus-vertex}

\end{enumerate}

Let $ch^*(x)$ denote the \emph{final charge} of an element $x$ of $V\cup F$ after discharging.
On one hand, from Euler's formula $|V|+|E|-|F|=2$, we deduce $\sum\limits_{x\in V\cup F}ch(x)=0.$
Since the sum of charges over all elements of $V\cup F$ is unchanged during the discharging precedure, it follows that $\sum\limits_{x\in V\cup F}ch^*(x)=0.$ On the other hand, we will show that $ch^*(x)\geq 0$ for $x\in V\cup F\setminus \{f_0\}$ and  $ch^*(f_0)> 0$. So, this obvious contradiction completes the proof of Theorem \ref{thm_main_extension}.

\begin{claim}
	$ch^*(f_0)>0$.
\end{claim}
\begin{proof}
	Notice that $R$\ref{term_rule_exterior_face} is the only rule making $f_0$ move charges out, charge 3 to each incident vertex.
	Recall that $ch(f_0)=d(f_0)+24$ and $d(f_0)\leq 11$. So, $ch^*(f_0)\geq ch(f_0)-3d(f_0)=24-2d(f)>0$.
\end{proof}

\begin{claim}
$ch^*(v)\geq0$ for $v\in V$.
\end{claim}
\begin{proof}
Denote by $m_3(v)$ the number of pendent 3-faces of $v$, and by $n_i(v)$ the number of $i$-faces containing $v$ for $i\in \{3,5\}$, where these countings excludes $f_0$. 
Since $G$ has no cycles of length 4 or 6, we have 
\begin{equation} \label{eq_counting} 
2n_3(v)+n_5(v)+m_3(v)\leq d(v).
\end{equation}
Furthermore, if $n_5(v)\notin \{0,d(v)\}$, then 
\begin{equation} \label{eq_counting_strict} 
2n_3(v)+n_5(v)+m_3(v)\leq d(v)-1.
\end{equation}
	
Case 1: first assume that $v$ is external. 
By $R$\ref{term_rule_exterior_face}, $v$ always receives charge 3 from $f_0$. 
Since $D$ is a cycle, $d(v)\geq 2$. 
If $d(v)=2$, then $v$ receives charge 1 from the other incident face by $R$\ref{term_rule_2-vertex}, giving $ch^*(v)=ch(v)+3+1=0$. 
Hence, we may next assume that $d(v)\geq 3$.
Denote by $f_1$ and $f_2$ the two ceiling faces containing $v$.
W.l.o.g., let $d(f_1)\leq d(f_2)$.

Case 1.1: suppose $d(v)=3$. In this case, $ch(v)=1$, and $v$ sends charge to $f_1$ and $f_2$ when $R$\ref{term_rule_exterior-3plus-vertex} is applicable to $v$. 
If $d(f_1)=3$, then on one hand, $d(f_2)\geq 7$, since $G$ has neither 4-cycles nor 6-cycles; on the other hand, $f_2$ is not a 3-ceiling 7-face by using Lemma \ref{lem_splitting path}. So $v$ sends to $f_2$ charge at most $\frac{1}{2}$, giving $ch^*(v)\geq ch(v)+3-\frac{7}{2}-\frac{1}{2}=0$.
We may next assume that $d(f_1)\geq 5$. 
Lemma \ref{lem_splitting path} also implies that not both $f_1$ and $f_2$ are 2-ceiling 5-faces. So, $v$ sends to $f_1$ and $f_2$ a total charge at most $\frac{13}{6}+\frac{3}{2}$, giving $ch^*(v)\geq ch(v)+3-\frac{13}{6}-\frac{3}{2}=\frac{1}{3}>0$.

Case 1.2: suppose $d(v)\geq 4$. $v$ sends charge out, only by $R$\ref{term_rule_exterior-3plus-vertex}, possibly to ceiling 3- or 5- or 7-faces, sticking 3- or 5-faces and pendent 3-faces.
So, 
\begin{equation} \label{eq_ext3vertex_charge}
ch^*(v)\geq \begin{cases}
ch(v)+3-\frac{7}{2}-\frac{7}{2}-6(n_3(v)-2)-\frac{8}{3}n_5(v)-\frac{5}{3}m_3(v)=ch(v)-\eta(v)+8,\text{ when $d(f_1)=d(f_2)=3$};\\
ch(v)+3-\frac{7}{2}-\frac{13}{6}-6(n_3(v)-1)-\frac{8}{3}(n_5(v)-1)-\frac{5}{3}m_3(v)=ch(v)-\eta(v)+7,\text{ when $d(f_1)=3$ and $d(f_2)=5$};\\
ch(v)+3-\frac{7}{2}-1-6(n_3(v)-1)-\frac{8}{3}n_5(v)-\frac{5}{3}m_3(v)=ch(v)-\eta(v)+\frac{9}{2},\text{ when $d(f_1)=3$ and $d(f_2)\geq 7$};\\
ch(v)+3-\frac{13}{6}-\frac{13}{6}-6n_3(v)-\frac{8}{3}(n_5(v)-2)-\frac{5}{3}m_3(v)=ch(v)-\eta(v)+4,\text{ when $d(f_1)=d(f_2)=5$};\\
ch(v)+3-\frac{13}{6}-1-6n_3(v)-\frac{8}{3}(n_5(v)-1)-\frac{5}{3}m_3(v)=ch(v)-\eta(v)+\frac{5}{2},\text{ when $d(f_1)=5$ and $d(f_2)\geq 7$};\\
ch(v)+3-1-1-6n_3(v)-\frac{8}{3}n_5(v)-\frac{5}{3}m_3(v)=ch(v)-\eta(v)+1,\text{ when $d(f_1)\geq 7$},
\end{cases} 
\end{equation}
where $\eta(v)=6n_3(v)+\frac{8}{3}n_5(v)+\frac{5}{3}m_3(v).$
Moreover, since $f_0$ is a face containing $v$,
Equation (\ref{eq_counting}) can be strengthen as:
\begin{equation}\label{eq_zeta}
\zeta(v)=2n_3(v)+n_5(v)+m_3(v)\leq \begin{cases}
d(v), \text{ when $d(f_1)=d(f_2)=3$};\\
d(v)-1, \text{ when either $d(f_1)=3$ and $d(f_2)\geq 5$ or $d(f_1)=d(f_2)=5$};\\
d(v)-2, \text{ when $d(f_1)\geq 5$ and $d(f_2)\geq 7.$}
\end{cases}
\end{equation}
Since $\eta(v)\leq 3\zeta(v)$, combining Equations (\ref{eq_ext3vertex_charge}) and (\ref{eq_zeta}) gives $ch^*(v)\geq ch(v)-3d(v)+7=2d(v)-7>0$.

Case 2: it remains to assume that $v$ is internal. By Lemma \ref{lem_min degree}, $d(v)\geq 3.$

Case 2.1: suppose that $d(v)=3$. In this case, $ch(v)=1$ and $n_3(v)\leq 1$.
Notice that only the rule $R$\ref{term_rule_3-vertex} makes $v$ send charge out.
So, if $v$ is triangular, $ch^*(v)=ch(v)-1=0$; otherwise, $ch^*(v)=ch(v)-\frac{1}{3}\times 3=0$. 

Case 2.2: suppose that $d(v)=4$. In this case, $ch(v)=6$.
Notice that, if $v$ is incident with no $(4,4,4)$-faces, then exactly three rules $R$\ref{term_rule_4-vertex_incident}, $R$\ref{term_rule_heavy_pendent} and $R$\ref{term_rule_5-face} make $v$ send charge out, to incident 3-faces, pendent 3-faces and incident 5-faces, respectively; otherwise, an additional rule $R$\ref{term_rule_normal-2-abnormal} is applied to $v$.
Clearly, $n_3(v)\leq 2$.
We distinguish three cases.

Case 2.2.1: assume that $n_3(v)= 0$. So, $m_3(v)+n_5(v)\leq 4.$
If $v$ has no pendent $(3,3,3)$-faces, then $v$ sends to each pendent 3-face or incident 5-face charge at most $\frac{3}{2}$, giving $ch^*(v)\geq ch(v)-\frac{3}{2}(m_3(v)+n_5(v))\geq 0$. So, we may assume that $v$ has a pendent $(3,3,3)$-face. It follows that  $n_5(v)\leq 2.$
By Lemma \ref{lem_2pendent}, $v$ has no other pendent $(3,3,3)$- or $(3,3,4)$-faces, which implies that $v$ sends to any other pendent 3-face charge at most $\frac{5}{4}$.
So, $ch^*(v)\geq ch(v)-\frac{5}{3}-\frac{3}{2}\times 2-\frac{5}{4}=\frac{1}{12}>0$.

Case 2.2.2: assume that $n_3(v)= 1$. In this case, either $n_5(v)=1$ and $m_3(v)=0$, or $n_5(v)=0$ and $m_3(v)\leq 2$.
For the former case, we have $ch^*(v)\geq ch(v)-\frac{7}{2}-\frac{3}{2}=1>0.$
For the latter case, we argue as follows. 
Denote by $f$ the 3-face containing $v$. 
If $f$ is a $(3,4^-,4)$-face, then $v$ has no pendent $(3,3,4^-)$-faces by Lemma \ref{lem_4vertex-1incident-1pendent}, giving $ch^*(v)\geq ch(v)-\frac{7}{2}-\frac{5}{4}\times 2=0.$ 
So, let us assume $f$ is not a $(3,4^-,4)$-face. By $R$\ref{term_rule_4-vertex_incident}, $v$ sends to $f$ charge at most $\frac{8}{3}$, and to abnormal vertices on $f$ a total charge at most $\frac{1}{6}\times 2$ when $R$\ref{term_rule_normal-2-abnormal} is applicable for $v$. Moreover, Combining Lemma \ref{lem_2pendent} and the rule $R$\ref{term_rule_heavy_pendent} yields that $v$ sends to possible pendent 3-faces a total charge at most $\max\{\frac{5}{3}+\frac{5}{4},\frac{3}{2}\times 2\}$, equal to 3.
Therefore, $ch^*(v)\geq ch(v)-\frac{8}{3}-\frac{1}{6}\times 2-3=0.$ 

Case 2.2.3: assume that $n_3(v)= 2$. So, $m_3(v)=n_5(v)=0$. 
Denote by $f_1$ and $f_2$ two 3-faces incident with $v$. 
If both $f_1$ and $f_2$ are not $(3,4,4)$-faces, then no matter $f_i$ has abnormal vertices or not, $v$ sends to $f_i$ and possiblely abnormal vertices on $f_i$ a total charge at most 3, giving $ch^*(v)\geq ch(v)-3\times 2=0.$ So, we may next assume that $f_1$ is a $(3,4,4)$-face. By $R$\ref{term_rule_4-vertex_incident}, $v$ sends charge $\frac{8}{3}$ to $f_1$. By Lemma \ref{lem_4vertex-2incident}, $f_2$ is not a $(3,4^-,4)$-face. If $f_2$ is further not a $(4,4,4)$-face, then $v$ sends to $f_2$ charge at most $\frac{5}{2}$, giving $ch^*(v)\geq ch(v)-\frac{7}{2}-\frac{5}{2}=0.$
So, we may further assume that $f_2$ is a $(4,4,4)$-face, that is, $v$ is abnormal.
If $f_2$ contains a normal vertex, then from it $v$ receives charge $\frac{1}{6}$ by $R$\ref{term_rule_normal-2-abnormal}, giving $ch^*(v)\geq ch(v)-\frac{7}{2}-\frac{8}{3}+\frac{1}{6}=0.$
So, we may assume that all the vertices on $f$ are abnormal. That is to say, $f_2$ together with three 3-faces intersecting with $f_2$ forms a wheel or an antiwheel, say $W$.
Since $G$ has no wheels by Lemma \ref{lem_wheel}, $W$ is an antiwheel.
By Lemma \ref{lem_antiwheel}, $W$ has a heavy outer neighbor, that is, $W$ has a strong $(3,4,4)$-face. By the rule $R$\ref{term_rule_face-2-abnormal}, $v$ receives charge $\frac{1}{6}$ from this face, giving $ch^*(v)= ch(v)-\frac{7}{2}-\frac{8}{3}+\frac{1}{6}=0.$

Case 2.3: suppose that $d(v)= 5$. In this case, $ch(v)=11$ and $n_3(v)\leq 2$.
Notice that only rules $R$\ref{term_rule_5^+-vertex_incident}, $R$\ref{term_rule_heavy_pendent} and $R$\ref{term_rule_5-face} make $v$ send charge out, to incident 3-faces, pendent 3-faces and incident 5-faces, respectively.
We distinguish three cases.

Case 2.3.1: assume that $n_3(v)=2$. So, $n_5(v)=0$ and $m_3(v)\leq 1$.
Denote by $f_1$ and $f_2$ the two 3-faces containing $v$ and by $f$ the pendent 3-face of $v$ if it exists.
If both $f_1$ and $f_2$ are not weak $(3,3,5)$-faces, then $v$ sends to each of them charge at most $\frac{9}{2}$, giving $ch^*(v)=ch(v)-\frac{9}{2}\times 2-\frac{5}{3}=\frac{1}{3}>0.$
So, we may assume that $v$ is incident with a weak $(3,3,5)$-face, say $f_1$. By Lemma \ref{lem_5vertex-2incident}, $f_2$ is neither a $(3,3,5)$-face nor a $(3,4,5)$-face. If $f_2$ is further not a weak $(3,5,5)$-face, then $v$ sends to $f_2$ charge $\frac{10}{3}$, giving $ch^*(v)=ch(v)-6-\frac{10}{3}-\frac{5}{3}=0$. So, let $f_2$ be a weak $(3,5,5)$-face. By Lemma \ref{lem_5vertex-2incident-1pendent}, $f$ is neither a $(3,3,3)$-face nor a $(3,3,4)$-face. So, $v$ sends to $f$ charge $\frac{5}{4}$, giving $ch^*(v)=ch(v)-6-\frac{7}{2}-\frac{5}{4}=\frac{1}{4}>0.$

Case 2.3.2: assume that $n_3(v)= 1$.
We can deduce that, $n_5(v)=2$ and $m_3(v)=0$, or $n_5(v)=1$ and $m_3(v)\leq 1$, or $n_5(v)=0$ and $m_3(v)\leq 3$. For the first case, Lemma \ref{lem_two-5-faces-adjacent} implies that not both 5-faces incident with $v$ are small. So $v$ sends to at least one of them charge $\frac{3}{2}$, giving $ch^*(v)\geq ch(v)-6-\frac{8}{3}-\frac{3}{2}=\frac{5}{6}>0$. 
For the latter two cases, a direct calculation gives 
$ch^*(v)\geq ch(v)-6-\frac{8}{3}-\frac{5}{3}=\frac{2}{3}>0$ and $ch^*(v)\geq ch(v)-6-\frac{5}{3}\times 3=0$, respectively. 

Case 2.3.3: assume that $n_3(v)= 0$.
Lemma \ref{lem_two-5-faces-adjacent} implies that $v$ has at most two small 5-faces around.
For any other incident 5-face or any pendent 3-face, $v$ sends to it charge no greater than $\frac{5}{3}$, giving $ch^*(v)\geq ch(v)-\frac{8}{3}\times 2-\frac{5}{3}(n_5(v)+m_3(v)-2)\geq \frac{2}{3}>0$, where Equation (\ref{eq_counting}) has been used for the second inequality.  

Case 2.4: suppose that $d(v)= 6$.  In this case, $ch(v)=16$ and
only rules $R$\ref{term_rule_6-vertex_incident}, $R$\ref{term_rule_heavy_pendent} and $R$\ref{term_rule_5-face} make $v$ send charge out, to incident 3-faces, pendent 3-faces and incident 5-faces, respectively.
If $n_5(v)=6$, then $ch^*(v)\geq ch(v)-\frac{8}{3}\times 6=0,$ we are done.
Moreover, if $n_5(v)\in \{1,2,\ldots,5\}$, then we have 
$ch^*(v)\geq ch(v)-6n_3(v)-\frac{8}{3} n_5(v)-\frac{5}{3} m_3(v)\geq ch(v)-\frac{2}{3}n_3(v)-\frac{8}{3}(2n_3(v)+n_5(v)+m_3(v))\geq 16-\frac{2}{3}n_3(v)-\frac{8}{3}(d(v)-1)=\frac{8}{3}-\frac{2}{3}n_3(v)>0$,
where the third inequality follows from Equation (\ref{eq_counting_strict}). 
Hence, we may next assume that $n_5(v)=0$. Analogously, by using Equation (\ref{eq_counting}) instead of Equation (\ref{eq_counting_strict}), we can deduce that $ch^*(v)\geq ch(v)-6n_3(v)-\frac{5}{3} m_3(v)\geq ch(v)-\frac{8}{3}n_3(v)-\frac{5}{3}(2n_3(v)+m_3(v))\geq 16-\frac{8}{3}n_3(v)-\frac{5}{3}d(v)=6-\frac{8}{3}n_3(v)>0$, provided by $n_3(v)\leq 2$.
Hence, we may next assume that $n_3(v)=3$. 
If $v$ is incident with at most one weak $(3,3,6)$-face, then $ch^*(v)\geq ch(v)-6-5\times 2=0$; 
otherwise, Lemma \ref{lem_6vertiex} implies that $v$ is incident with a 3-face $f$ that is neither $(3,3,6)$-face nor $(3,4,6)$-face. So, $v$ sends to $f$ charge 4, giving $ch^*(v)\geq ch(v)-6\times 2-4=0$.

Case 2.5: suppose that $d(v)\geq 7$. In this case, $v$ sends to any incident 3-face charge 6 by $R$\ref{term_rule_7-vertex_incident}, to any incident 5-face charge at most $\frac{8}{3}$ by $R$\ref{term_rule_5-face}, and to any pendent 3-face charge at most $\frac{5}{3}$ by $R$\ref{term_rule_heavy_pendent}.
So, $ch^*(v)\geq ch(v)-6n_3(v)-\frac{8}{3} n_5(v)-\frac{5}{3} m_3(v)\geq ch(v)-3(2n_3(v)+n_5(v)+m_3(v))\geq (5d(v)-14)-3d(v)\geq 0$, where the last two inequalities follow from Equation (\ref{eq_counting}) and the assumption $d(v)\geq 7$, respectively.
\end{proof}

\begin{claim}
$ch^*(f)\geq 0$ for $f\in F\setminus \{f_0\}$.
\end{claim}
\begin{proof}
Since $G$ has neither 4-cycles nor 6-cycles, $d(f)\notin \{4,6\}$.

Case 1: assume that $f$ contains vertices of $D$. 
Denote by $n_2(f)$ the number of 2-vertices on $f$.
Lemma \ref{lem_splitting path} implies that, if $d(f)\in \{3,5,7\}$ then the common part of $f$ and $D$ must be a path of length at most $\frac{d(f)-1}{2}$, say the path $P$.
Here, a path of length 0 or 1 means a vertex or an edge, respectively. So, $n_2(f)\leq \frac{d(f)-1}{2}-1$. We distinguish four cases.

Case 1.1: let $d(f)= 3$. In this case, $ch(f)=-8$ and $P$ is either a vertex or an edge.
Notice that $f$ receives charge at least 1 from each incident internal vertex by rules from $R$\ref{term_rule_3-vertex} to $R$\ref{term_rule_7-vertex_incident}.
If $P$ is a vertex, then $f$ is a sticking 3-face, which receives charge 6 from $P$ by $R$\ref{term_rule_exterior-3plus-vertex}, giving $ch^*(f)=ch(f)+6+1\times 2=0$, we are done.
If $P$ is an edge, then $f$ is a 1-ceiling 3-face, which receives charge $\frac{7}{2}$ from both vertices of $P$ by $R$\ref{term_rule_exterior-3plus-vertex}, giving $ch^*(f)=ch(f)+\frac{7}{2}\times 2 + 1=0,$ we are done as well.

Case 1.2: let $d(f)= 5$. By $R$\ref{term_rule_3-vertex} and $R$\ref{term_rule_5-face}, $f$ receives from each exterior vertex of $f$ charge at least $\frac{1}{3}$.
Clearly, $ch(f)=-4$ and $P$ is a vertex or an edge or a 2-path. If $P$ is a vertex, then $f$ receives charge $\frac{8}{3}$ from this vertex by $R$\ref{term_rule_exterior-3plus-vertex}, giving $ch^*(f)=ch(f)+\frac{1}{3}\times 4+ \frac{8}{3}=0$.
If $P$ is an edge, then $f$ receives charge $\frac{3}{2}$ from both vertices of $P$ by $R$\ref{term_rule_exterior-3plus-vertex}, giving $ch^*(f)=ch(f)+\frac{1}{3}\times 3+ \frac{3}{2}\times 2=0$.
If $P$ is a 2-path, then $f$ receives charge $\frac{13}{6}$ from each end vertex of $P$ by $R$\ref{term_rule_exterior-3plus-vertex} and sends charge 1 to the unique 2-vertex of $P$, giving $ch^*(f)=ch(f)+\frac{1}{3}\times 2+ \frac{13}{6}\times 2-1=0$.
We are done in all the three situations above.

Case 1.3: let $d(f)= 7$. In this case, $f$ sends charge to incident 2-vertices by $R$\ref{term_rule_2-vertex} and receives charge from incident exterior $3^+$-vertices by $R$\ref{term_rule_exterior-3plus-vertex}, no other charges moving about $f$.
Recall that $ch(f)=2d(f)-14=0$ and $n_2(f)\leq \frac{d(f)-1}{2}-1=2$. If $n_2(f)=2$, i.e., $f$ is a 3-ceiling face, then $f$ receives charge $1$ from each end vertex of $P$, giving $ch^*(f)=ch(f)+1\times 2-1\times n_2(f)=0.$
If $n_2(f)=1$, i.e., $f$ is a 2-ceiling face, then $f$ receives charge $\frac{1}{2}$ from each end vertex of $P$, giving $ch^*(f)=ch(f)+\frac{1}{2}\times 2-1\times n_2(f)=0.$
If $n_2(f)=0$, then $f$ has no charges moving in or out, giving $ch^*(f)=ch(f)=0.$
We are done in all the three situations above.

Case 1.4: let $d(f)\geq 8$. 
Since $f$ is not $f_0$, $f$ contains an internal vertex. That is to say,
$f$ contains a splitting $2^+$-path of $D$, say $Q$.
By Lemma \ref{lem_splitting path}, if $|Q|\leq 4$, then $Q$ divides $D$ into two parts, one of which together with $Q$ forms a face. Now $Q$ contains internal 2-vertices, contradicting Lemma \ref{lem_min degree}. So, $|Q|\geq 5$. It follows that $n_2(f)\leq d(f)-6$.
By our discharging rules, $8^+$-faces send charge only to incident 2-vertices, charge 1 to each by $R$\ref{term_rule_2-vertex}.
So, $ch^*(f)=ch(f)-1\times n_2(f)\geq (2d(f)-14)-(d(f)-6)=d(f)-8\geq 0.$

Case 2: assume that $f$ is vertex-disjoint with $D$. We distinguish three cases.

Case 2.1: let $d(f)\geq 7$. By our discharging rules, $f$ has no charges moved in or out in this case. So, $ch^*(f)=ch(f)=2d(f)-14\geq 0.$

Case 2.2: let $d(f)=5$. In this case, $ch(f)=-4$. By our discharging rules, $f$ sends no charges out and receives from each incident $4^+$-vertex charge at least $\frac{1}{3}$ by $R$\ref{term_rule_3-vertex} or $R$\ref{term_rule_5-face}. By Lemma \ref{lem_5-face-all-light}, $f$ contains a $4^+$-vertex, say $u$.
If $u$ is the only $4^+$-vertex on $f$, i.e., $f$ is small, then Lemma \ref{lem_5-face-4-light} implies that $u$ is further a $5^+$-vertex, which sends to $f$ charge $\frac{8}{3}$ by $R$\ref{term_rule_5-face}, giving  $ch^*(f)\geq ch(f)+\frac{8}{3}+\frac{1}{3}\times 4=0;$
otherwise, $f$ has at least two $4^+$-vertices, from each $f$ receives charge $\frac{3}{2}$,  giving  $ch^*(f)\geq ch(f)+\frac{3}{2}\times 2+\frac{1}{3}\times 3=0.$ 

Case 2.3: let $d(f)=3$. In this case, $ch(f)=-8$ and $f$ receives charge from all the incident vertices and from all heavy outer neighbors, and sends charge out only when $R$\ref{term_rule_face-2-abnormal} applied. In particular, $f$ receives charge $1$ from each incident 3-vertex by $R$\ref{term_rule_3-vertex}.

If $f$ is a $(3,3,3)$-face, then Lemma \ref{lem_light_triangle} implies that $f$ has three heavy outer neighbors, each sends charge $\frac{5}{3}$ to $f$ by $R$\ref{term_rule_heavy_pendent} or $R$\ref{term_rule_exterior-3plus-vertex}. 
So, $ch^*(f)=ch(f)+\frac{5}{3}\times 3+1\times 3=0.$

If $f$ is a $(3,3,4)$-face, then $f$ has precisely two heavy outer neighbors by Lemma \ref{lem_light_triangle}, each sends charge at least $\frac{3}{2}$ to $f$ by $R$\ref{term_rule_heavy_pendent} or $R$\ref{term_rule_exterior-3plus-vertex}. Moreover, $f$ receives charge 3 from the 4-vertex of $f$ by $R$\ref{term_rule_4-vertex_incident}. So, $ch^*(f)=ch(f)+\frac{3}{2}\times 2+3+1\times 2=0.$ 

If $f$ is a weak $(3,3,5)$-face or a weak $(3,3,6)$-face or a $(3,3,7^+)$-face, then $f$ receives charge 6 from the $5^+$-vertex of $f$ by $R$\ref{term_rule_5^+-vertex_incident} or $R$\ref{term_rule_6-vertex_incident} or $R$\ref{term_rule_7-vertex_incident}, respectively. So, $ch^*(f)=ch(f)+6+1\times 2=0.$ 

If $f$ is a strong $(3,3,5)$-face or a strong $(3,3,6)$-face, then $f$ receives charge at least $\frac{7}{2}$ from the $5^+$-vertex of $f$ by $R$\ref{term_rule_5^+-vertex_incident} or $R$\ref{term_rule_6-vertex_incident}, respectively. Moreover, $f$ receives charge at least $\frac{5}{4}$ from both heavy outer neighbors of $f$ by $R$\ref{term_rule_heavy_pendent} or $R$\ref{term_rule_exterior-3plus-vertex}. So, $ch^*(f)\geq ch(f)+\frac{7}{2}+\frac{5}{4}\times 2+1\times 2=0.$ 

If $f$ is a weak $(3,4,4)$-face or a weak $(3,5,5)$-face, then $f$ receives charge $\frac{7}{2}$ from both 4-vertices or 5-vertices of $f$ by $R$\ref{term_rule_4-vertex_incident} or $R$\ref{term_rule_5^+-vertex_incident}, respectively. So, $ch^*(f)=ch(f)+\frac{7}{2}\times 2+1=0.$ 

If $f$ is a strong $(3,4,4)$-face, then $f$ might send charge out by $R$\ref{term_rule_face-2-abnormal}.
Notice that $f$ is contained in at most two antiwheels, that is, $f$ sends charge to at most six abnormal vertices, charge $\frac{1}{6}$ to each. Moreover, since $f$ is strong, $f$ has a heavy outer neighbor, from which $f$ receives charge at least $\frac{5}{4}$ by $R$\ref{term_rule_heavy_pendent} or $R$\ref{term_rule_exterior-3plus-vertex}.
So, $ch^*(f)\geq ch(f)-\frac{1}{6}\times 6+\frac{5}{4}+\frac{7}{2}\times 2+1=\frac{1}{4}>0.$

If $f$ is a $(3,4,5^+)$-face, then $f$ receives charge $\frac{5}{2}$ from the 4-vertex of $f$ by $R$\ref{term_rule_4-vertex_incident} and charge at least $\frac{9}{2}$ from the $5^+$-vertex of $f$ by $R$\ref{term_rule_5^+-vertex_incident} or $R$\ref{term_rule_6-vertex_incident} or $R$\ref{term_rule_7-vertex_incident}.
So, $ch^*(f)\geq ch(f)+\frac{5}{2}+\frac{9}{2}+1=0.$ 

If $f$ is a strong $(3,5,5)$-face, then $f$ receives charge at least $\frac{5}{4}$ from the heavy outer neighbor by $R$\ref{term_rule_heavy_pendent} or $R$\ref{term_rule_exterior-3plus-vertex} and charge $\frac{7}{2}$ from both 5-vertices of $f$ by $R$\ref{term_rule_5^+-vertex_incident}.
So, $ch^*(f)\geq ch(f)+\frac{5}{4}+\frac{7}{2}\times 2+1=\frac{1}{4}>0.$ 

If $f$ is a $(3,5^+,6^+)$-face, then $f$ receives charge 3 and charge at least 4 from the $5^+$-vertex and the $6^+$-vertex on $f$, respectively.
So, $ch^*(f)\geq ch(f)+3+4+1=0.$

If $f$ is a $(4,4,4)$-face, then $f$ receives charge $\frac{8}{3}$ from each incident vertex by $R$\ref{term_rule_4-vertex_incident}, giving $ch^*(f)= ch(f)+\frac{8}{3}\times 3=0.$

If $f$ is a $(4,4^+,5^+)$-face, then $f$ receives charge $\frac{5}{2}$, charge at least $\frac{5}{2}$ and charge at least $3$ from the $4$-vertex, the $4^+$-vertex and the $5^+$-vertex, respectively.
So, $ch^*(f)\geq ch(f)+\frac{5}{2}+\frac{5}{2}+3=0.$
\end{proof}

By the previous three claims, the proof of Theorem \ref{thm_main_extension} is completed.

\section{Acknowledgement}

The first author is supported by Zhejiang Provincial Natural Science Foundation of China (ZJNSF), Grant Number: LY20A010014.
The second author is supported by National Natural Science Foundation of China (NSFC), Grant Number: 11901258.

\end{document}